\documentclass[11pt]{amsart}

\usepackage{latexsym}
\usepackage{amsmath}
\usepackage{amssymb}

\usepackage{mathtools}

\usepackage{enumerate}

\usepackage[mathscr]{eucal}

\usepackage{cleveref}

\newtheorem{theorem}{Theorem}[section]
\newtheorem{lemma}[theorem]{Lemma}
\newtheorem{corollary}[theorem]{Corollary}
\newtheorem{proposition}[theorem]{Proposition}
\newtheorem{conjecture}[theorem]{Conjecture}

\newcommand{\cl}{{\rm cl}}

\usepackage{color}
\newcommand{\blue}{\textcolor{black}}

\newcommand{\bloo}{\textcolor{black}}

\usepackage{tikz}
\usetikzlibrary{calc}

\DeclarePairedDelimiter\ceil{\lceil}{\rceil}
\DeclarePairedDelimiter\floor{\lfloor}{\rfloor}

\begin{document}

\title{Cyclic matroids}

\thanks{The first and second authors were supported by the New Zealand Marsden Fund. The first author was also supported by a Rutherford Foundation Postdoctoral Fellowship.  The third author was supported by a University of Canterbury Doctoral Scholarship}

\author{Nick Brettell}
\address{School of Mathematics and Statistics, Victoria University of Wellington, Wellington, New Zealand}
\email{nick.brettell@vuw.ac.nz}

\author{Charles Semple}
\address{School of Mathematics and Statistics, University of Canterbury, Christchurch, New Zealand}
\email{charles.semple@canterbury.ac.nz}

\author{Gerry Toft}
\address{School of Mathematics and Statistics, University of Canterbury, Christchurch, New Zealand}
\email{gerry.toft@pg.canterbury.ac.nz}



\date{\today}

\begin{abstract}
For integers $s$ and $t$ exceeding one, a matroid~$M$ on $n$ elements is {\em nearly $(s, t)$-cyclic} if there is a cyclic ordering $\sigma$ of its ground set such that every $s-1$ consecutive elements of $\sigma$ are contained in an $s$-element circuit and every $t-1$ consecutive elements of $\sigma$ are contained in a $t$-element cocircuit. In the case $s=t$, nearly $(s, s)$-cyclic matroids have been studied previously. In this paper, we show that if $M$ is nearly $(s, t)$-cyclic and $n$ is sufficiently large, then these $s$-element circuits and $t$-element cocircuits are consecutive in $\sigma$ in a prescribed way, that is, $M$ is ``$(s, t)$-cyclic''. Furthermore, we show that, given $s$ and $t$ where $t\ge s$, every $(s, t)$-cyclic matroid on $n > s+t-2$ elements is a weak-map image of the $\left(\frac{t-s}{2}\right)$-th truncation of a certain $(s, s)$-cyclic matroid. If $s=3$, this certain matroid is the rank-$\frac{n}{2}$ whirl, and if $s=4$, this certain matroid is the rank-$\frac{n}{2}$ free swirl.
\end{abstract}

\maketitle

\section{Introduction}

Tutte's Wheels-and-Whirls Theorem~\cite{tut66} is synonymous with matroid theory. It says that, except for wheels and whirls, every $3$-connected matroid has a single-element deletion or a single-element contraction that is $3$-connected. The reason for this exception is that wheels and whirls are precisely the $3$-connected matroids in which every element is in a $3$-element circuit and a $3$-element cocircuit. In fact, wheels and whirls have a stronger property: if $M$ is a wheel or a whirl, then there is a cyclic ordering $\sigma$ of its ground set such that every \bloo{set of} two consecutive elements in $\sigma$ is contained in a $3$-element circuit and a $3$-element cocircuit. Furthermore, if $M$ is a wheel and $r(M)\ge 4$, or if $M$ is a whirl and $r(M)\ge 3$, then these $3$-element circuits and $3$-element cocircuits are unique, and the elements of these $3$-element circuits and $3$-element cocircuits are consecutive in $\sigma$. Brettell et al.~\cite{bre19b} studied matroids satisfying a generalisation of this property, that is, for a positive integer $s$ exceeding one, matroids whose ground sets have a cyclic ordering $\sigma$ such that every \bloo{set of} $s-1$ consecutive elements in $\sigma$ is contained in an $s$-element circuit and an $s$-element cocircuit. In this paper, we extend this study by considering generalisations of these matroids whereby the size of the circuit and the size of the cocircuit need not be the same.

Let $s$ and $t$ be positive integers exceeding one. A matroid $M$ is {\em nearly $(s, t)$-cyclic} if there exists a cyclic ordering $\sigma$ of $E(M)$ such that every \bloo{set of} $s-1$ consecutive elements of $\sigma$ \bloo{is} contained in an $s$-element circuit and every \bloo{set of} $t-1$ consecutive elements of $\sigma$ \bloo{is} contained in a $t$-element cocircuit, in which case we say that $\sigma$ is a {\em nearly $(s, t)$-cyclic ordering} of $E(M)$. \bloo{Although not explicitly stated, there is an implicit assumption that if $M$ is nearly $(s,t)$-cyclic, then $M$ has at least $\max\{s,t\}-1$ elements, so it has at least one $s$-element circuit and at least one $t$-element cocircuit.}

Wheels and whirls are nearly $(3, 3)$-cyclic, while spikes and swirls are nearly $(4, 4)$-cyclic. For all $r\ge 3$, a {\em rank-$r$ spike} is a matroid $M$ on $2r$ elements whose ground set can be partitioned into pairs $\{L_1, L_2, \ldots, L_r\}$ such that, for all distinct $i, j\in \{1, 2, \ldots, r\}$, the union of $L_i$ and $L_j$ is a $4$-element circuit and a $4$-element cocircuit. Therefore, if $\sigma$ is a cyclic ordering of $E(M)$ such that, for all $i$, the two elements in $L_i$ are consecutive in $\sigma$, then $\sigma$ is a nearly $(4, 4)$-cyclic ordering of $E(M)$. For all $r\ge 3$, a {\em rank-$r$ swirl} is a matroid $M$ on $2r$ elements obtained by first taking a simple matroid whose ground set is the disjoint union of a basis $B=\{b_1, b_2, \ldots, b_r\}$ and $2$-element sets $L_1, L_2, \ldots, L_r$ such that $L_i\subseteq \cl(\{b_i, b_{i+1}\})$ for all $i\in \{1, 2, \ldots, r\}$, where subscripts are interpreted modulo $r$, and then deleting the elements in $B$. If, for all $i$, the elements in $L_i$ are freely placed in the span of $\{b_i, b_{i+1}\}$ in this construction, then the resulting matroid is the {\em rank-$r$ free swirl}. Observe that $L_i\cup L_{i+1}$ is $4$-element circuit and a $4$-element cocircuit for all $i$. Therefore, if $L_i=\{e_i, f_i\}$ for all $i$, then $\sigma=(e_1, f_1, e_2, f_2, \ldots, e_r, f_r)$ is a nearly $(4, 4)$-cyclic ordering of $E(M)$, and so $M$ is nearly $(4, 4)$-cyclic.

The examples of nearly $(s, t)$-cyclic matroids in the last paragraph all have the property that $s=t$. To see an example of a nearly $(s, t)$-cyclic matroid where $s\neq t$, take a sufficiently large whirl and truncate it, that is freely add an element $f$ to the whirl, and then contract $f$. It is not difficult to show that the resulting matroid is nearly $(3, 5)$-cyclic. More generally, given \blue{odd} $t\ge 3$, the $\left(\frac{t-3}{2}\right)$-th truncation of a sufficiently large whirl results in a matroid that is nearly $(3, t)$-cyclic (see \Cref{free}).

Nearly $(s, t)$-cyclic matroids are highly structured. For example, suppose that $M$ is a rank-$r$ wheel, where $r\ge 4$, and $\sigma=(e_1, e_2, \ldots, e_n)$ is a nearly $(3, 3)$-cyclic ordering of its ground set. \bloo{Then, for all $i\in \{1, 2, \ldots, n\}$, one of $\{e_i,e_{i+1},e_{i+2}\}$ and $\{e_{i-1},e_i,e_{i+1}\}$ is the unique $3$-element circuit containing $\{e_i,e_{i+1}\}$ and the other is the unique $3$-element cocircuit containing $\{e_i,e_{i+1}\}$, with the parity of $i$ determining which is the circuit and which is the cocircuit.} The following definition captures this structure.

Let $s$ and $t$ be positive integers exceeding one. A matroid $M$ is {\em $(s, t)$-cyclic} if there exists a cyclic ordering $\sigma=(e_1, e_2, \ldots, e_n)$ of $E(M)$
such that each of the following holds, where subscripts are interpreted modulo $n$:
\begin{enumerate}[(i)]
\item either $\{e_1, e_2, \ldots, e_s\}$ or $\{e_2, e_3, \ldots, e_{s+1}\}$ is an $s$-element circuit of $M$;

\item either $\{e_1, e_2, \ldots, e_t\}$ or $\{e_2, e_3, \ldots, e_{t+1}\}$ is a $t$-element cocircuit of $M$;

\item if $\{e_i, e_{i+1}, \ldots, e_{i+s-1}\}$ is an $s$-element circuit for some $i\in \{1, 2, \ldots, n\}$, then $\{e_{i+2}, e_{i+3}, \ldots, e_{i+s+1}\}$ is also an $s$-element circuit of $M$; and

\item if $\{e_i, e_{i+1}, \ldots, e_{i+t-1}\}$ is a $t$-element cocircuit for some $i~\in~\left\{1, 2, \ldots, n\right\}$, then $\{e_{i+2}, e_{i+3}, \ldots, e_{i+t+1}\}$ is also a $t$-element cocircuit of $M$.
\end{enumerate}
A cyclic ordering satisfying (i)--(iv) is called an {\em $(s, t)$-cyclic ordering} of $E(M)$. \bloo{Note that our terminology differs from ~\cite{bre19b}; what we call a nearly $(t,t)$-cyclic ordering of a matroid $M$ was previously called a cyclic $(t-1,t)$-ordering of $M$, and what we call a $(t,t)$-cyclic ordering of $M$ was previously called a $t$-cyclic ordering of $M$.}

If $M$ is nearly $(2, 2)$-cyclic, then, as noted in~\cite{bre19b}, $M$ is obtained by taking direct sums of copies of $U_{1, 2}$, and so $M$ is $(2, 2)$-cyclic. Brettell et al.~\cite[Theorem~1.1]{bre19b} showed that, for all $s\ge 3$, if $\sigma$ is a nearly $(s, s)$-cyclic ordering of a matroid $M$ on $n$ elements and $n\ge 6s-10$, then $\sigma$ is an $(s, s)$-cyclic ordering of $M$. The first main result of this paper generalises that theorem.

\begin{theorem}
Let $M$ be a matroid on $n$ elements, and suppose that $\sigma$ is a nearly $(s, t)$-cyclic ordering of $M$, where $s, t\ge 3$. Let $t_1=\min\{s, t\}$ and $t_2=\max\{s, t\}$. If $n\ge 3t_1+t_2-5$ and $n\ge t_1+2t_2-1$, then $\sigma$ is an $(s, t)$-cyclic ordering of $M$.
\label{par_to_tot}
\end{theorem}


\noindent The proof of Theorem~\ref{par_to_tot} takes a different approach to that used in~\cite{bre19b}. Equating $s$ and $t$ in Theorem~\ref{par_to_tot}, we have the following corollary, improving the lower bound in~\cite[Theorem~1.1]{bre19b}.

\begin{corollary}
Let $M$ be a matroid on $n$ elements, and suppose that $\sigma$ is a nearly $(s, s)$-cyclic ordering of $M$ \blue{for $s\ge 3$. If $n \geq \max\{8, 4s-5\}$, then} $\sigma$ is an $(s, s)$-cyclic ordering of $M$.
\end{corollary}


For all positive integers $s$ and $t$ exceeding one, we will show that if a matroid on $n$ elements is nearly $(s, t)$-cyclic, then $n\ge s+t-2$. Observe that, for all such $s$ and $t$, the uniform matroid $U_{s-1, s+t-2}$ is nearly $(s, t)$-cyclic with \blue{$s+t-2$} elements. Thus this lower bound is sharp. Furthermore, if a matroid on $n$ elements is $(s, t)$-cyclic and $n>s+t-2$, then we will also show that $n$ is even and $s\equiv t\bmod{2}$. Hence, if a matroid $M$ is $(s, t)$-cyclic and $s\not\equiv t\bmod{2}$, then $M$ has exactly $s+t-2$ elements. Lastly, we suspect the inequalities $n\ge 3t_1+t_2-5$ and $n\ge t_1+2t_2-1$ in Theorem~\ref{par_to_tot} are
not tight, and leave it as an open problem to determine, for all positive integers $s, t\ge 2$, tight lower bounds on the size of the ground set of a matroid $M$ having the property that if $\sigma$ is a nearly $(s, t)$-cyclic ordering of $E(M)$, then $\sigma$ is an $(s, t)$-cyclic ordering of $E(M)$.

The second main result of this paper, Theorem~\ref{free}, shows that $(s, t)$-cyclic matroids are not wild. In particular, this result shows that, given positive integers $s$ and $t$ exceeding one, such that $t\ge s$, an $(s, t)$-cyclic matroid $M$ on $n$ elements, where \bloo{$n > s+t-2$}, is a weak-map image of the $\left(\frac{t-s}{2}\right)$-th truncation of a certain $(s, s)$-cyclic matroid. To formally state Theorem~\ref{free}, let $M_1$ and $M_2$ be matroids on ground sets $E_1$ and $E_2$, respectively, and suppose that $|E_1|=|E_2|$. Let $\varphi:E_1\rightarrow E_2$ be a bijection. We say $\varphi$ is a {\em weak map} from $M_1$ to $M_2$ if, for all independent sets $I$ in $M_2$, the set $\varphi^{-1}(I)$ is independent in $M_1$. Equivalently, $\varphi$ is a weak map from $M_1$ to $M_2$ if, for all circuits $C$ of $M_1$, the set $\varphi(C)$ contains a circuit of $M_2$. If $\varphi$ is such a map, $M_2$ is a {\em weak-map image} of $M_1$, and $M_1$ is said to be {\em freer} than $M_2$.

\begin{figure}
	\centering
	\begin{tikzpicture}
	\coordinate (a) at (30:1);
	\coordinate (b) at (90:1);
	\coordinate (c) at (150:1);
	\coordinate (d) at (210:1);
	\coordinate (e) at (270:1);
	\coordinate (f) at (330:1);
	
	\coordinate (a1) at (-10:2.5);
	\coordinate (a2) at (10:2.5);
	\coordinate (b1) at (50:2.5);
	\coordinate (b2) at (70:2.5);
	\coordinate (c1) at (110:2.5);
	\coordinate (c2) at (130:2.5);
	\coordinate (d1) at (170:2.5);
	\coordinate (d2) at (190:2.5);
	\coordinate (e1) at (230:2.5);
	\coordinate (e2) at (250:2.5);
	\coordinate (f1) at (290:2.5);
	\coordinate (f2) at (310:2.5);
	
	\draw (a1) -- (a) -- (a2);
	\draw (b1) -- (a) -- (b2);
	\draw (b1) -- (b) -- (b2);
	\draw (c1) -- (b) -- (c2);
	\draw (c1) -- (c) -- (c2);
	\draw (d1) -- (c) -- (d2);
	\draw (d1) -- (d) -- (d2);
	\draw (e1) -- (d) -- (e2);
	\draw (e1) -- (e) -- (e2);
	\draw (f1) -- (e) -- (f2);
	\draw (f1) -- (f) -- (f2);
	\draw (a1) -- (f) -- (a2);
	
	\draw[fill=white] (a) circle (2pt) node [label=left:$2$, xshift=4] {};
	\draw[fill=white] (a1) circle (2pt) node [label=right:$e_6$, xshift=-4] {};
	\draw[fill=white] (a2) circle (2pt) node [label=right:$e_5$, xshift=-4] {};
	\draw[fill=white] (b) circle (2pt) node [label=below:$1$, yshift=3] {};
	\draw[fill=white] (b1) circle (2pt) node [label=above right:$e_4$, xshift=-5, yshift=-5] {};
	\draw[fill=white] (b2) circle (2pt) node [label=above right:$e_3$, xshift=-5, yshift=-5] {};
	\draw[fill=white] (c) circle (2pt) node [label=right:$6$, xshift=-4] {};
	\draw[fill=white] (c1) circle (2pt) node [label=above left:$e_2$, xshift=5, yshift=-5] {};
	\draw[fill=white] (c2) circle (2pt) node [label=above left:$e_1$, xshift=5, yshift=-5] {};
	\draw[fill=white] (d) circle (2pt) node [label=right:$5$, xshift=-4] {};
	\draw[fill=white] (d1) circle (2pt) node [label=left:$e_{12}$, xshift=4] {};
	\draw[fill=white] (d2) circle (2pt) node [label=left:$e_{11}$, xshift=4] {};
	\draw[fill=white] (e) circle (2pt) node [label=above:$4$, yshift=-3] {};
	\draw[fill=white] (e1) circle (2pt) node [label=below left:$e_{10}$, xshift=5, yshift=5] {};
	\draw[fill=white] (e2) circle (2pt) node [label=below left:$e_9$, xshift=5, yshift=5] {};
	\draw[fill=white] (f) circle (2pt) node [label=left:$3$, xshift=4] {};
	\draw[fill=white] (f1) circle (2pt) node [label=below right:$e_8$, xshift=-5, yshift=5] {};
	\draw[fill=white] (f2) circle (2pt) node [label=below right:$e_7$, xshift=-5, yshift=5] {};
	\end{tikzpicture}
	\caption{The bipartite graph $G^{12}_4$.}
	\label{mostfreecycliceg}
\end{figure}
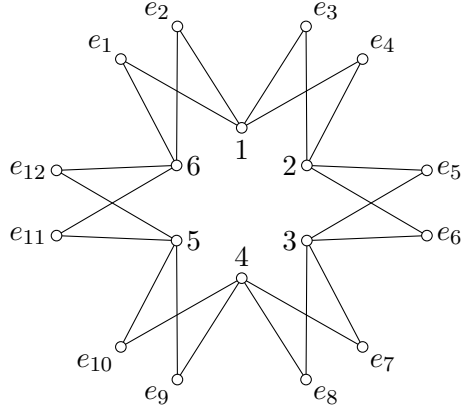

For vertices $u$ and $v$ of a graph, $u$ is a {\em neighbour} of $v$ if $u$ is adjacent to $v$, \bloo{and we let $N(v)$ denote the set of neighbours of $v$. Note that here, as well as elsewhere in the paper, we adopt the convention of writing singletons without set braces provided there is no ambiguity.}

Now let $s$ be an integer exceeding one and let $n$ be a positive even integer. We next define a certain matroid with parameters $s$ and $n$ \blue{that is transversal and co-transversal}. Let $G^n_s$ be the bipartite graph with vertex parts $E=\{e_1, e_2, \ldots, e_n\}$ and $\{1, 2, \ldots, \frac{n}{2}\}$ and, for all $i\in \{1, 2, \ldots, \frac{n}{2}\}$, the set of neighbours of $i$ is
$$N(i)=\{e_{2i-1}, e_{2i}, \ldots, e_{2i+s-2}\},$$
where subscripts are interpreted modulo $n$. For example, if $n = 12$ and $s = 4$, then $G^{12}_4$ is the bipartite graph shown in~\Cref{mostfreecycliceg}. The transversal matroid on $E$ in which
\bloo{$$(N(1), N(2), \ldots, N(\textstyle{\frac{n}{2}}))$$}
is a presentation is an example of a multi-path matroid~\cite{bon07}. Denote the dual of this transversal matroid by $\Psi^n_s$. \bloo{Multi-path matroids have the property that their duals are transversal~\cite[Theorem~3.8]{bon07}, so $\Psi^n_s$ is a transversal matroid. In fact, we shall show that $\Psi^n_s$ is a self-dual matroid.} If $s=2$, then $\Psi^n_s$ is isomorphic to the rank-$\frac{n}{2}$ matroid obtained by taking direct sums of copies of $U_{1, 2}$; while if $s=3$ \blue{or} $s=4$, then $\Psi^n_s$ is isomorphic to the rank-$\frac{n}{2}$ whirl \blue{or} rank-$\frac{n}{2}$ free swirl, respectively. For example, the dual of the transversal matroid realised by $G^{12}_4$ is the rank-$6$ free swirl. More generally, it turns out that, for all $s\ge 2$, the matroid $\Psi^n_s$ is $(s, s)$-cyclic.

Let $M$ be a matroid. If $r(M)>0$, then the matroid obtained from $M$ by freely adding an element $f$ and then contracting $f$ is called the {\em truncation} of $M$ and is denoted by $T(M)$. If $r(M)=0$, we set $T(M)=M$. For all positive integers $i$, the {\em $i$-th truncation} of $M$, denoted $T^i(M)$, is defined iteratively as $T^i(M)=T(T^{i-1}(M))$, where $T^0(M)=M$. The second main result of this paper is the following theorem.

\begin{theorem}
Let $M$ be an $(s, t)$-cyclic matroid on $n$ elements, where \blue{$n~\geq~s+t-1$}. If $t\ge s$, then $M$ is a weak-map image of the $\left(\frac{t-s}{2}\right)$-th truncation of $\Psi^n_s$, an $(s, t)$-cyclic matroid.
\label{free}
\end{theorem}

In addition to this paper and \cite{bre19b}, \blue{there have been several recent studies into matroids with particular prescribed circuits and cocircuits}. Miller~\cite{mil14} investigated the matroids in which every pair of elements is contained in a $4$-element circuit and a $4$-element cocircuit, while Oxley et al.~\cite{oxl19} considered the $3$-connected matroids in which every pair of elements is in a $4$-element circuit and every element is in a $3$-element cocircuit, and the $4$-connected matroids in which every pair of elements is contained in a $4$-element circuit and a $4$-element cocircuit. Furthermore, Brettell et al.~\cite{bre19a} studied matroids in which every $t$-element subset of the ground set is contained in an $\ell$-element circuit and an $\ell$-element cocircuit. Relevant to this paper, \blue{their results imply} that if a matroid $M$ has the property that every $t$-element subset of $E(M)$ is contained in a $2t$-element circuit and a $2t$-element cocircuit, then, provided $|E(M)|$ is sufficiently large, $M$ is $(2t, 2t)$-cyclic. Further results concerning $(3, t)$-cyclic and $(4, t)$-cyclic matroids, including a characterisation of the $(4, 4)$-cyclic matroids on at least $8$ elements, will be found in Gerry Toft's PhD thesis.


The paper is organised as follows. The next section contains some preliminaries, while Section~\ref{props} establishes some basic properties of cyclic matroids. These properties are used in the proofs of Theorems~\ref{par_to_tot} and~\ref{free} which are given in Sections~\ref{cyclicthm} and~\ref{freethm}, respectively. The proof of Theorem~\ref{free} follows from a more general result concerning the duals of multi-path matroids. Lastly, in Section~\ref{counterexample}, we give a counterexample to a conjecture concerning $(s, s)$-cyclic matroids, given in~\cite{bre19b}. This conjecture says that if $s$ is an integer exceeding two \blue{and $M$ is an $(s,s)$-cyclic matroid, then $M$ can} be obtained from either a wheel or a \bloo{whirl} (if $s$ is odd), or either a spike or a swirl (if $s$ is even) by a sequence of elementary quotients and elementary lifts. Unless otherwise specified, notation and terminology follows~\cite{ox11}.

\section{Preliminaries}
\label{prelims}

Throughout the paper, we say two sets $X$ and $Y$ {\em intersect} if $X\cap Y$ is non-empty; otherwise, $X$ and $Y$ {\em do not intersect}. For a positive integer $m$, we let $[m]$ denote the set $\{1, 2, \ldots, m\}$. Furthermore, \bloo{for} $i, j\in [m]$, we let $[i, j]$ denote the set $\{i, i+1, \ldots, j\}$ if $i\le j$ and the set $\{i, i+1, \ldots, m, 1, 2, \ldots, j\}$ if \bloo{$i > j$}. Now let $\sigma=(e_1, e_2, \ldots, e_n)$ be a cyclic ordering of $\{e_1, e_2, \ldots, e_n\}$. For all $i, j\in [n]$, the notation \bloo{$\sigma(i, j)$} denotes the set of elements $\{e_i, e_{i+1}, \ldots, e_j\}$, where subscripts are interpreted modulo $n$.

The following well-known lemma is used frequently in the paper. \bloo{The phrase {\em by orthogonality} signals an application of this lemma.}

\begin{lemma}
Let $M$ be a matroid. If $C$ is a circuit and $C^*$ is a cocircuit of $M$, then $|C\cap C^*|\neq 1$.
\end{lemma}

The next lemma concerns the independent sets of the $i$-th truncation of a matroid (see, for example, \cite[Proposition~7.3.10]{ox11}).

\begin{lemma}
Let $M$ be a matroid with $r(M)\ge 1$, and let $i$ be a non-negative integer such that $i\le r(M)$. Then
$$\mathcal I(T^i(M))=\{X\in \mathcal I(M): |X|\le r(M)-i\}.$$
\label{truncation1}
\end{lemma}

\section{Properties of Cyclic Matroids}
\label{props}

In this section, we establish various properties of nearly $(s, t)$-cyclic and $(s, t)$-cyclic matroids on $n$ elements. \bloo{The first lemma is used frequently in this section.}

\begin{lemma}
\label{adj_cocircuits}
\bloo{Let $M$ be an $(s, t)$-cyclic matroid on $n$ elements, where $n~>~s+t-2$, and let $\sigma = (e_1,e_2,\ldots,e_n)$ be an $(s,t)$-cyclic ordering of $M$. Then,
\begin{enumerate}[{\rm (i)}]
	\item if $\{e_i,e_{i+1},\ldots,e_{i+s-1}\}$ is a circuit, then $\{e_{i-t},e_{i-t+1},\ldots,e_{i-1}\}$ and $\{e_{i+s},e_{i+s+1},\ldots,e_{i+s+t-1}\}$ are cocircuits, and
	\item if $\{e_i,e_{i+1},\ldots,e_{i+t-1}\}$ is a cocircuit, then $\{e_{i-s},e_{i-s+1},\ldots,e_{i-1}\}$ and $\{e_{i+t},e_{i+t+1},\ldots,e_{i+s+t-1}\}$ are circuits.
\end{enumerate}}
\end{lemma}

\begin{proof}
\bloo{We will prove (i). The proof of (ii) follows by duality as $M^*$ is a $(t,s)$-cyclic matroid. Since $\sigma$ is an $(s,t)$-cyclic ordering of $M$, it follows that one of $\{e_{i-t},e_{i-t+1},\ldots,e_{i-1}\}$ and $\{e_{i-t+1},e_{i-t+2},\ldots,e_i\}$ is a $t$-element cocircuit of $M$. But, as $n > s+t-2$, the set $\{e_{i-t+1},e_{i-t+2},\ldots,e_i\}$ intersects the circuit $\{e_i,e_{i+1},\ldots,e_{i+s-1}\}$ in one element, and so $\{e_{i-t+1},e_{i-t+2},\ldots,e_i\}$ is not a cocircuit. Therefore $\{e_{i-t},e_{i-t+1},\ldots,e_{i-1}\}$ is a cocircuit of $M$.} Similarly, $\{e_{i+s-1},e_{i+s},\ldots,e_{i+s+t-2}\}$ is not a cocircuit as it intersects $\{e_i,e_{i+1},\ldots,e_{i+s-1}\}$ in one element, and so $\{e_{i+s},e_{i+s+1},\ldots,e_{i+s+t-1}\}$ is a cocircuit.
\end{proof}

The next two lemmas consider the relationships amongst $s$, $t$, and $n$.

\begin{lemma}
\label{lower_bound}
Let $M$ be a nearly $(s, t)$-cyclic matroid on $n$ elements. Then $n\geq s + t - 2$.
\end{lemma}

\begin{proof}
Since $M$ contains an $s$-element circuit, we have that $r(M) \geq s-1$. Similarly, as $M$ contains a $t$-element cocircuit, $r^*(M) \geq t-1$. Therefore, as $n = r(M) + r^*(M)$, we also have that $n \geq s+t-2$.
\end{proof}

Note that the bound in \Cref{lower_bound} is tight. In particular, for any positive integers $s, t\ge 2$, the uniform matroid $U_{s-1, s+t-2}$ is nearly $(s, t)$-cyclic. In fact, $U_{s-1, s+t-2}$ is $(s, t)$-cyclic.

\begin{lemma}
\label{even_n}
Let $M$ be an $(s, t)$-cyclic matroid on $n$ elements. If $n\left.>s+t-2\right.$, then
\begin{enumerate}[{\rm (i)}]
\item $n$ is even, and

\item $s\equiv t \bmod{2}$.
\end{enumerate}
\end{lemma}

\begin{proof}
Suppose $n > s+t-2$. To prove (i), assume that $n$ is odd. Let $\left.\sigma = (e_1, e_2, \ldots, e_n)\right.$ be an $(s, t)$-cyclic ordering of $M$, and let $\{e_i, e_{i+1}, \ldots, e_{i+s-1}\}$ be a circuit of $M$. Then, for all even $k$, the set $\{e_{i+k}, e_{i+1+k}, \ldots, e_{i+s-1+k}\}$ is a circuit of $M$. \bloo{In particular, taking $k=n-1$, the set $\{e_{i-1}, e_i, \ldots, e_{i+s-2}\}$ is a circuit of $M$. But, by \Cref{adj_cocircuits}, the set $\{e_{i-t},e_{i-t+1},\ldots,e_{i-1}\}$ is a cocircuit of $M$, and, since $n > s+t-2$, this cocircuit intersects the circuit $\{e_{i-1},e_i,\ldots,e_{i+s-2}\}$ in one element. This contradiction implies $n$ is even.}

For the proof of (ii), assume that $s \not\equiv t \bmod{2}$. Let $\sigma=(e_1, e_2, \ldots, e_n)$ be an $(s, t)$-cyclic ordering of $M$, and let $\{e_i, e_{i+1}, \ldots, e_{i+s-1}\}$ be a circuit of $M$. \bloo{By \Cref{adj_cocircuits}, the set $\{e_{i-t},e_{i-t+1},\ldots,e_{i-1}\}$ is a cocircuit of $M$. By the assumption, $s+t-1$ is even and so, as $(i-t)+(s+t-1)=i+s-1$, the set $\{e_{i+s-1},e_{i+s},\ldots,e_{i+s+t-2}\}$ is a cocircuit. But this cocircuit intersects $\{e_i, e_{i+1}, \ldots, e_{i+s-1}\}$ in precisely one element, contradicting orthogonality.} Therefore, $s \equiv t \bmod{2}$, completing the proof of (ii).
\end{proof}

The bound in \Cref{even_n} is tight. For example, choosing one of $s$ and $t$ to be even and the other to be odd, the uniform matroid $U_{s-1, s+t-2}$ is an $(s, t)$-cyclic matroid on $s+t-2$ elements. However, \Cref{even_n} shows that there is no $(s, t)$-cyclic matroid with more elements.

Generalising~\cite[Lemma~4.3, Lemma~5.1, Lemma~5.3]{bre19b}, the next four lemmas concern the independent sets, closure operator, and rank function of $(s, t)$-cyclic matroids. A consequence of the first of these lemmas is that if $s=t$ and $s$ is even, then the $s$-element circuits and $s$-element cocircuits in an $(s, s)$-cyclic ordering of a matroid coincide. \bloo{On the other hand}, if $s=t$ and $s$ is odd, then the $s$-element circuits and $s$-element cocircuits in an $(s, s)$-cyclic ordering of a matroid behave like the $3$-element circuits and $3$-element cocircuits in $(3, 3)$-cyclic orderings of whirls.

\begin{lemma}
\label{structure}
Let $M$ be an $(s, t)$-cyclic matroid on $n$ elements, where $\left.n > s+t-2\right.$, and let $\sigma=(e_1, e_2, \ldots, e_n)$ be an $(s, t)$-cyclic ordering of $M$. Suppose that $\{e_i, e_{i+1}, \ldots, e_{i+s-1}\}$ is a circuit of $M$. If $s$ and $t$ are even, then
\begin{enumerate}[{\rm (i)}]
\item $\{e_i, e_{i+1}, \ldots, e_{i+t-1}\}$ is a cocircuit,

\item $\{e_{i+1}, e_{i+2}, \ldots, e_{i+s}\}$ is independent, and

\item $\{e_{i+1}, e_{i+2}, \ldots, e_{i+t}\}$ is coindependent. 
\end{enumerate}
Furthermore, if $s$ and $t$ are odd, then
\begin{enumerate}
\item[{\rm (iv)}] $\{e_{i+1},e_{i+2}, \ldots, e_{i+t}\}$ is a cocircuit,

\item[{\rm (v)}] $\{e_{i+1}, e_{i+2}, \ldots, e_{i+s}\}$ is independent, and

\item[{\rm (vi)}] $\{e_i,e_{i+1},\ldots,e_{i+t-1}\}$ is coindependent.
\end{enumerate}
\end{lemma}

\begin{proof}
\bloo{By \Cref{adj_cocircuits}, the set $\{e_{i-t},e_{i-t+1},\ldots,e_{i-1}\}$ is a cocircuit of $M$.} If $t$ is even, this implies $\{e_i,e_{i+1},\ldots,e_{i+t-1}\}$ is a cocircuit; otherwise, $t$ is odd and $\{e_{i+1},e_{i+2},\ldots,e_{i+t}\}$ is a cocircuit. 	
	
We next show that $\{e_{i+1},e_{i+2},\ldots,e_{i+s}\}$ is independent. Suppose this is not the case. Then $\{e_{i+1}, e_{i+2}, \ldots, e_{i+s}\}$ contains a circuit, call it $C$. \bloo{By \Cref{adj_cocircuits}, the set $\{e_{i+s}, e_{i+s+1}, \ldots, e_{i+s+t-1}\}$ is a cocircuit of $M$}. Therefore, if $e_{i+s} \in C$, then $C$ intersects $\{e_{i+s},e_{i+s+1},\ldots,e_{i+s+t-1}\}$ in exactly one element, a contradiction. But if $e_{i+s} \notin C$, then $C$ is properly contained in the circuit $\{e_i,e_{i+1},\ldots,e_{i+s-1}\}$, another contradiction. Thus, no such circuit $C$ exists, and so $\{e_{i+1},e_{i+2},\ldots,e_{i+s}\}$ is independent. \bloo{We have shown that, if $\{e_i,e_{i+1},\ldots,e_{i+s-1}\}$ is a circuit, then $\{e_{i+1},e_{i+2},\ldots,e_{i+s}\}$ is independent. Since $M^*$ is a $(t,s)$-cyclic matroid, this implies that if $\{e_j,e_{j+1},\ldots,e_{j+t-1}\}$ is a cocircuit, then $\{e_{j+1},e_{j+2},\ldots,e_{j+1}\}$ is coindependent. This is sufficient to show (iii) and (vi) and complete the proof.}
\end{proof}

\begin{lemma}
\label{closure}
Let $M$ be an $(s, t)$-cyclic matroid on $n$ elements, where $n > s+t-2$, and let $\sigma = (e_1, e_2, \ldots, e_n)$ be an $(s, t)$-cyclic ordering of $M$. Then, for all $i \in [n]$ and \blue{$s-1 \leq k \leq n-t$},
\begin{enumerate}[{\rm (i)}]
\item $e_{i+k} \in \cl(\{e_i, e_{i+1}, \ldots, e_{i+k-1}\})$ if and only if
\[\{e_{i+k-s+1}, e_{i+k-s+2}, \ldots, e_{i+k}\}\]
is a circuit, and
\item $e_{i-1} \in \cl(\{e_i, e_{i+1}, \ldots, e_{i+k-1}\})$ if and only if 
\[\{e_{i-1}, e_i, \ldots, e_{i+s-2}\}\]
 is a circuit.
\end{enumerate}
\end{lemma}


\begin{proof}
We will prove (i). \bloo{Then (ii) follows from the fact that reversing the order of $\sigma$ gives another $(s,t)$-cyclic ordering of $M$}. Since $k \geq s-1$, if $\{e_{i+k-s+1},e_{i+k-s+2},\ldots,e_{i+k}\}$ is a circuit, then $e_{i+k} \in \cl(\{e_i,e_{i+1},\ldots,e_{i+k-1}\})$. Conversely, suppose $e_{i+k} \in \cl(\{e_i,e_{i+1},\ldots,e_{i+k-1}\})$. Then there exists a circuit $C$ contained in $\{e_i,e_{i+1},\ldots,e_{i+k}\}$ such that $C$ contains $e_{i+k}$. Assume $\{e_{i+k-s+1},e_{i+k-s+2},\ldots,e_{i+k}\}$ is not a circuit. If $s$ and $t$ are even, then, by \Cref{structure}, the set $\{e_{i+k-s},e_{i+k-s+1},\ldots,e_{i+k-s+t-1}\}$ is a cocircuit and \bloo{so}, as $s$ is even, the set $\{e_{i+k},e_{i+k+1},\ldots,e_{i+k+t-1}\}$ is also a cocircuit. Since $k \leq n-t$, this last cocircuit intersects $C$ only in the element $e_{i+k}$, a contradiction. Therefore, $\{e_{i+k-s+1},e_{i+k-s+2},\ldots,e_{i+k}\}$ is a circuit. Similarly, if $s$ and $t$ are odd, then, by \Cref{structure}, $\{e_{i+k-s+1}, e_{i+k-s+2}, \ldots, e_{i+k-s+t}\}$ is a cocircuit, which means $\{e_{i+k},e_{i+k+1},\ldots,e_{i+k+t-1}\}$ is also a cocircuit. Again, this contradicts orthogonality with $C$, showing that $\{e_{i+k-s+1},e_{i+k-s+2},\ldots,e_{i+k}\}$ is a circuit, and completing the proof of the lemma.
\end{proof}

\begin{lemma}
\label{set_rank}
Let $M$ be an $(s, t)$-cyclic matroid on $n$ elements, where $\left.n > s+t-2\right.$, and let $\sigma = (e_1, e_2, \ldots, e_n)$ be an $(s, t)$-cyclic ordering of $M$. Then, for all $i \in [n]$ and $1\le k \leq n-t+1$,
\begin{align*}
r(\{e_i, e_{i+1}, \ldots, & e_{i+k-1}\}) =
\begin{cases}
k, & \mbox{if $k < s$}; \\
\floor*{\frac{s+k-1}{2}}, & \mbox{if $k\geq s$ and $\{e_i, e_{i+1}, \ldots, e_{i+s-1}\}$} \\
& \qquad \mbox{is a circuit}; \\
\ceil*{\frac{s+k-1}{2}}, & \mbox{if $k\geq s$ and $\{e_i, e_{i+1}, \ldots, e_{i+s-1}\}$} \\
& \qquad \mbox{is not a circuit}.
\end{cases}
\end{align*}
\end{lemma}

\begin{proof}
The proof is by induction on $k$. If $k < s$, then $\{e_i, e_{i+1}, \ldots, e_{i+k-1}\}$ is \bloo{a proper subset of an $s$-element circuit, so it is independent. Therefore,} $r(\{e_i, e_{i+1}, \ldots, e_{i+k-1}\}) = k$. \bloo{Now} suppose $k=s$. If $\{e_i, e_{i+1}, \ldots, e_{i+s-1}\}$ is a circuit, then
$$r(\{e_i, e_{i+1}, \ldots, e_{i+s-1}\})=s-1=\floor*{{\textstyle \frac{s+s-1}{2}}},$$
while if $\{e_i, e_{i+1}, \ldots, e_{i+s-1}\}$ is not a circuit, then, by \Cref{closure},
$$r(\{e_i, e_{i+1}, \ldots, e_{i+s-1}\})=s=\ceil*{{\textstyle \frac{s+s-1}{2}}}.$$
Thus the lemma holds for all $1\le k\le s$.

Now suppose that $s+1\le k\le n-t+1$, and the lemma holds for the set $\{e_i, e_{i+1}, \ldots, e_{i+k-2}\}$. Consider $\{e_i, e_{i+1}, \ldots, e_{i+k-1}\}$. First assume that $\{e_i, e_{i+1}, \ldots, e_{i+s-1}\}$ is a circuit. If $s+k$ is odd, then $k-s$ is odd, and it follows by Lemma~\ref{structure}(ii) and (v) that $\{e_{i+k-s}, e_{i+k-s+1}, \ldots, e_{i+k-1}\}$ is not a circuit. Therefore, by Lemma~\ref{closure}, $e_{i+k-1}\not\in \cl(\{e_i, e_{i+1}, \ldots, e_{i+k-2}\})$, and so, by the induction assumption,
\begin{align*}
r(\{e_i, e_{i+1}, \ldots, e_{i+k-1}\}) & = r(\{e_i, e_{i+1}, \ldots, e_{i+k-2}\}) + 1 \\
& = \floor*{{\textstyle \frac{s+k-2}{2}}}+1 = \floor*{{\textstyle \frac{s+k}{2}}} = \floor*{{\textstyle \frac{s+k-1}{2}}}
\end{align*}
as $s+k$ is odd. If $s+k$ is even, then $\{e_{i+k-s}, e_{i+k-s+1}, \ldots, e_{i+k-1}\}$ is a circuit, and so $e_{i+k-1}\in \cl(\{e_i, e_{i+1}, \ldots, e_{i+k-2}\})$. Therefore
\begin{align*}
r(\{e_i, e_{i+1}, \ldots, e_{i+k-1}\}) & = r(\{e_i, e_{i+1}, \ldots, e_{i+k-2}\}) \\
& = \floor*{{\textstyle \frac{s+k-2}{2}}} = \floor*{{\textstyle \frac{s+k-1}{2}}}
\end{align*}
as $s+k$ is even.

Now assume that $\{e_i, e_{i+1}, \ldots, e_{i+s-1}\}$ is not a circuit. If $s+k$ is odd, then $\{e_{i+k-s}, e_{i+k-s+1}, \ldots,e_{i+k-1}\}$ is a circuit, and so, by the induction assumption \blue{and \Cref{closure},}
\begin{align*}
r(\{e_i, e_{i+1}, \ldots, e_{i+k-1}\}) & = r(\{e_i, e_{i+1}, \ldots, e_{i+k-2}\}) \\
& = \ceil*{{\textstyle \frac{s+k-2}{2}}} =  \ceil*{{\textstyle \frac{s+k-1}{2}}}
\end{align*}
as $s+k$ is odd. If $s+k$ is even, then $\{e_{i+k-s}, e_{i+k-s+1}, \ldots, e_{i+k-1}\}$ is not a circuit, and so, by \Cref{closure} and the induction assumption,
\begin{align*}
r(\{e_i, e_{i+1}, \ldots, e_{i+k-1}\}) & = r(\{e_i, e_{i+1}, \ldots, e_{i+k-2}\}) + 1 \\
& = \ceil*{{\textstyle \frac{s+k-2}{2}}} + 1 = \ceil*{{\textstyle \frac{s+k}{2}}} = \ceil*{{\textstyle \frac{s+k-1}{2}}}
\end{align*}
as $s+k$ is even. This completes the proof of the lemma.
\end{proof}

The next lemma shows that the rank of an $(s, t)$-cyclic matroid on $n$ elements is invariant under $s$, $t$, and $n$.

\begin{lemma}
\label{mat_rank}
Let $M$ be an $(s, t)$-cyclic matroid on $n$ elements. Then $\left.r(M) = \frac{n+s-t}{2}\right.$ and $r^*(M) = \frac{n-s+t}{2}$.
\end{lemma}

\begin{proof}
\bloo{By \Cref{lower_bound}, the matroid $M$ has at least $s+t-2$ elements. Since $M$ has an $s$-element circuit and a $t$-element cocircuit, $r(M) \geq s-1$ and $r^*(M) \geq t-1$. Therefore, if $n=s+t-2$, then}
$$r(M) = s-1 = {\textstyle \frac{(s+t-2) + s-t}{2}}$$
and
$$r^*(M) = t-1 = {\textstyle \frac{(s+t-2)-s+t}{2}}.$$
\bloo{Otherwise, by \Cref{set_rank}, the set $\{e_1,e_2,\ldots,e_{n-t+1}\}$ either has rank $\floor*{\frac{n+s-t}{2}}$ or rank $\ceil*{\frac{n+s-t}{2}}$. By \Cref{even_n}, we have that $\frac{n+s-t}{2}$ is even, so
\[
r(\{e_1,e_2,\ldots,e_{n-t+1}\}) = \frac{n+s-t}{2}.
\]
Therefore, $r(M) \geq \frac{n+s-t}{2}$. Similarly, by Lemmas~\ref{even_n} and~\ref{set_rank}, we get that
\[
r^*(\{e_1,e_2,\ldots,e_{n-s+1}\}) = \frac{n-s+t}{2},
\]
and so $r^*(M) \geq \frac{n-s+t}{2}$. Since $\frac{n+s-t}{2} + \frac{n-s+t}{2} = n$, it follows that $\left.r(M) = \frac{n+s-t}{2}\right.$ and $r^*(M) = \frac{n-s+t}{2}$.}
\end{proof}

The last lemma in this section \blue{will be used to prove \Cref{par_to_tot} in the next section; we include it here as it may be of independent interest.}

\begin{lemma}
\label{total_circuits}
Let $s$ and $t$ be positive integers exceeding one, and let $\sigma = (e_1, e_2, \ldots, e_n)$ be a nearly $(s, t)$-cyclic ordering of a matroid $M$, where $n \geq s+t$. If $\{e_i, e_{i+1}, \ldots, e_{i+s-1}\}$ is a circuit for all odd $i \in [n]$, then $\sigma$ is an $(s, t)$-cyclic ordering of $M$.
\end{lemma}

\begin{proof}
It is sufficient to prove that, for all odd $i \in [n]$, the set $\{e_{i-t+2}, e_{i-t+3}, \ldots, e_{i+1}\}$ is a cocircuit. Consider the set $\{e_{i-t+2}, e_{i-t+3}, \ldots, e_i\}$. This set contains $t-1$ consecutive elements of $\sigma$, so must be contained in a $t$-element cocircuit $C^*$. Let $e_j$ be the unique element of $C^*$ not contained in $\{e_{i-t+2}, e_{i-t+3}, \ldots, e_i\}$. If $e_j \notin \{e_{i+1}, e_{i+2}, \ldots, e_{i+s-1}\}$, then $C^*$ intersects the circuit $\{e_i, e_{i+1}, \ldots, e_{i+s-1}\}$ in exactly one element, contradicting orthogonality. Furthermore, if $e_j \in \{e_{i+2}, e_{i+3}, \ldots, e_{i+s-1}\}$, then, as $n \geq s+t$, the cocircuit $C^*$ intersects the circuit $\{e_{i+2}, e_{i+3}, \ldots, e_{i+s+1}\}$ in exactly one element. This last contradiction implies that $e_j = e_{i+1}$, completing the proof of the lemma.
\end{proof}

\section{Proof of Theorem $1.1$}
\label{cyclicthm}

This section consists of the proof of Theorem~\ref{par_to_tot}. Throughout the section, let $M$ be a nearly $(s, t)$-cyclic matroid, where $s, t\ge 3$, and let $\sigma = (e_1, e_2, \ldots, e_n)$ be a nearly $(s, t)$-cyclic ordering of $M$. We shall prove that, provided $n$ is sufficiently large, $\sigma$ is an $(s, t)$-cyclic ordering of $M$.

Recall that, for all $i, j\in [n]$, we define $\sigma(i, j)$ to be the set $\{e_i, e_{i+1}, \ldots, e_j\}$. Additionally, for all $i\in [n]$, let $C_i$ be an arbitrarily chosen circuit of size $s$ containing $\sigma(i, i+s-2)$ and let $C^*_i$ be an arbitrarily chosen cocircuit of size $t$ containing $\sigma(i, i+t-2)$. There is a unique element of $C_i$ not contained in $\sigma(i, i+s-2)$; call this element $c_i$. Likewise, let $c^*_i$ be the unique element of $C^*_i$ not contained in $\sigma(i, i+t-2)$.

\begin{lemma}
\label{consec_neq}
If $n \geq s+2t-4$, then $c_i \neq c_{i+1}$ for all $i \in [n]$.
\end{lemma}

\begin{proof}
Suppose \blue{$n \geq s+2t-4$ and} $c_i = c_{i+1}$ for some $i\in [n]$. Then $C_i = \sigma(i, i+s-2) \cup \{c_i\}$ and $C_{i+1} = \sigma(i+1, i+s-1) \cup \{c_i\}$. By circuit elimination, there is a circuit, say $C$, of $M$ contained in $\sigma(i, i+s-1)$. If $C$ does not contain $e_i$, then $C$ is properly contained in the circuit $C_{i+1}$, a contradiction. Similarly, if $C$ does not contain $e_{i+s-1}$, then $C$ is properly contained in $C_i$. Therefore, $C$ contains both $e_i$ and $e_{i+s-1}$.

\bloo{Since $t \geq 2$, we have that $n \geq s+t-2$. Therefore, } the $(t-1)$-element set $\sigma\left(i+s-1, i+s+t-3\right)$ intersects $C$ in only the element $e_{i+s-1}$. Therefore, by orthogonality, \bloo{$c^*_{i+s-1} \in C - \{e_{i+s-1}\} \subseteq \sigma(i, i+s-2)$}. This means that $C^*_{i+s-1}$ and $\sigma\left(i,i+s-2\right)$ also intersect in exactly one element. Therefore, by orthogonality, \bloo{$c_i \in \sigma(i+s-1,i+s+t-3)$}.

Similarly, the $(t-1)$-element set $\sigma(i-t+2,i)$ intersects $C$ in only the element $e_i$. Therefore, orthogonality between $C^*_{i-t+2}$ and $C$ implies that \bloo{$c^*_{i-t+2} \in C-\{e_i\} \subseteq \sigma(i+1,i+s-1)$}. Applying orthogonality again, this time between $C^*_{i-t+2}$ and $C_{i+1}$, shows that \bloo{$c_{i+1} \in \sigma(i-t+2,i)$}. But $c_i = c_{i+1}$, and so $c_i$ is contained in both $\sigma(i-t+2,i)$ and $\sigma(i+s-1,i+s+t-3)$, two sets which are disjoint since $n \geq s+2t-4$. This contradiction implies that $c_i \neq c_{i+1}$ and completes the proof.
\end{proof}

\bloo{The next lemma is used several times in the proof of \Cref{circ_unique}.}

\begin{lemma}
\label{circ_unique_helper}
\bloo{Suppose there exists $d_i \neq c_i$ such that $D_i = \sigma(i,i+s-2) \cup \{d_i\}$ is a circuit of $M$. Let $j \in [n]$ such that $|\sigma(j,j+t-2) \cap \{c_i,d_i\}| = 1$. Then $\sigma(j,j+t-2)$ intersects $\sigma(i,i+s-2)$.}
\end{lemma}

\begin{proof}
\bloo{Without loss of generality, we may assume that \bloo{$c_i \in \sigma(j,j+t-2)$} and \bloo{$d_i \notin \sigma(j,j+t-2)$}. Suppose $\sigma(j,j+t-2)$ does not intersect $\sigma(i,i+s-2)$. Then $\sigma(j,j+t-2)$ intersects $C_i$ in one element. Therefore, by orthogonality, \bloo{$c_j^* \in \sigma(i,i+s-2)$}. But now $c_j^* \in D_i$, so $C_j^*$ and $D_i$ intersect in one element. This contradiction to orthogonality implies that $\sigma(j,j+t-2)$ intersects $\sigma(i,i+s-2)$, and completes the proof.}
\end{proof}

\begin{lemma}
\label{circ_unique}
If $n \geq s+2t-4$, then, for all $i \in [n]$, there is a unique circuit of size $s$ containing $\sigma(i, i+s-2)$.
\end{lemma}

\begin{proof}
We know $C_i$ is an $s$-element circuit containing $\sigma(i, i+s-2)$. Suppose that there is a second such circuit. This means that there is an element $d_i$, distinct from $c_i$, such that $\sigma(i, i+s-2)\cup \{d_i\}$ is a circuit. Call this circuit $D_i$.

Now, for some \bloo{$j \in [n]$}, we have $c_i = e_j$. Consider the $(t-1)$-element subsets $\sigma(j-t+2, j)$ and $\sigma(j, j+t-2)$. \bloo{Since $c_i \neq d_i$, at least one of these sets does not contain $d_i$. Up to symmetry, we may assume that $d_i \notin \sigma(j-t+2,j)$. Now, $|\sigma(j-t+2,j) \cap \{c_i,d_i\}|=1$ and so, by \Cref{circ_unique_helper}, the set $\sigma(j-t+2,j)$ intersects $\sigma(i,i+s-2)$. Since $n \geq s+2t-5$, this implies that $\sigma(j,j+t-2)$ does not intersect $\sigma(i,i+s-2)$. Applying \Cref{circ_unique_helper} again, we see that $|\sigma(j,j+t-2) \cap \{c_i,d_i\}| \neq 1$, so $d_i \in \sigma(j,j+t-2)$. Therefore, $\sigma(j+1,j+t-1)$ contains $d_i$ but does not contain $c_i$. However, since $n \geq s+2t-4$ and $\sigma(j-t+2,j)$ intersects $\sigma(i,i+s-2)$, we also have that $\sigma(j+1,j+t-1)$ is disjoint from $\sigma(i,i+s-2)$. This contradiction to \Cref{circ_unique_helper} shows that no such $d_i$ exists, thereby completing the proof.}
\end{proof}

\begin{lemma}
\label{consec_nearby}
Let $i, j\in [n]$ such that $c_i\in \sigma(j+1, j+t-2)$, and suppose that $n\geq 2s+t-4$. Then each of the following holds:
\begin{enumerate}[{\rm (i)}]
\item If $\sigma(j, j+t-1)$ does not intersect $\sigma(i, i+s-1)$, then $c_{i+1}\in \sigma(j, j+t-1)$.

\item If $\sigma(j, j+t-1)$ does not intersect $\sigma(i-1, i+s-2)$, then $\left.c_{i-1}\in \sigma(j, j+t-1)\right.$.
\end{enumerate}
\end{lemma}

\begin{proof}
We prove (i). \bloo{Then (ii) follows by reversing the order of $\sigma$}. Suppose that $\sigma(j, j+t-1)$ does not intersect $\sigma(i, i+s-1)$. Assume that \bloo{$c_{i+1}\notin \sigma(j, j+t-1)$}, and consider the $(t-1)$-element sets $\sigma(j, j+t-2)$ and $\sigma(j+1, j+t-1)$. Each of these sets contains $c_i$ and does not contain $c_{i+1}$. Furthermore, since $\sigma(j, j+t-1)$ and $\sigma(i, i+s-1)$ are disjoint, each of $\sigma(j, j+t-2)$ and $\sigma(j+1, j+t-1)$ intersects $C_i$ in exactly one element and does not intersect $C_{i+1}$. Therefore, by orthogonality, $c^*_j$ and $c^*_{j+1}$ are both contained in $C_i$, but not contained in $C_{i+1}$. The only possibility is $c^*_j = c^*_{j+1} = e_i$. However, this contradicts \Cref{consec_neq} when applied to $M^*$. Therefore, $c_{i+1} \in \sigma(j, j+t-1)$.
\end{proof}

\begin{lemma}
\label{consecutive}
Let $i\in [n]$, and suppose that $c_i = e_j$. If $n \geq s+2t-2$ and $n \geq 2s+t-4$, then at least one of the following holds:
\begin{enumerate} [{\rm (i)}]
\item $c_i$ and $c_{i+1}$ are both contained in $\sigma(i-1,i+s)$;

\item $c_{i+1} = e_{j+1}$; or

\item $c_{i+1} = e_{j-1}$.
\end{enumerate}
\end{lemma}

\begin{proof}
Suppose (i) does not hold, that is, at least one of $c_i$ and $c_{i+1}$ is not contained in $\sigma(i-1, i+s)$. \bloo{Choose $k \in [n]$ such that $e_k \in \{c_i,c_{i+1}\}$ and $e_k \notin \sigma(i-1, i+s)$}. Let $e_{k'}$ be the other element of $c_i$ and $c_{i+1}$. We establish the lemma by proving that either $k' = k+1$ or $k' = k-1$, which we shall do using \Cref{consec_nearby}.

First assume that \bloo{$e_k \notin \sigma(i-t+2, i+s+t-3)$}. This means that neither $\sigma(k-1, k+t-2)$ nor $\sigma(k-t+2, k+1)$ intersect $\sigma(i, i+s-1)$. So, by \Cref{consec_nearby} (using part (i) if $e_k = c_i$ or part (ii) if $e_k = c_{i+1}$), \bloo{we have that $e_{k'} \in \sigma(k-1, k+t-2) \cap \sigma(k-t+2, k+1)$}. Now, 
\[\sigma(k-1, k+t-2) \cap \sigma(k-t+2, k+1) = \{e_{k-1}, e_k, e_{k+1}\}\]
 and, by \Cref{consec_neq}, $e_{k'}\neq e_k$. Therefore, either $e_{k'} = e_{k-1}$ or $e_{k'} = e_{k+1}$, the desired result.

Now assume that \bloo{$e_k \in \sigma(i-t+2,i+s+t-3)$}. Then, as \bloo{$e_k \notin \sigma(i-1,i+s)$}, either \bloo{$e_k \in \sigma(i+s+1, i+s+t-3)$} or \bloo{$e_k \in \sigma(i-t+2,i-2)$}. We consider only the former case; the analysis for the latter case is symmetrical. \bloo{Thus, suppose $e_k \in \sigma(i+s+1, i+s+t-3)$}. Now, $\sigma(k-1, k+t-2)$ does not intersect $\sigma(i, i+s-1)$, as $k$ is at most $i+s+t-3$ and $n \geq s+2t-2$. Therefore, by \Cref{consec_nearby}, we have that \bloo{$e_{k'} \in \sigma(k-1,k+t-2)$}. If $e_{k'} \neq e_{k-1}$ and $e_{k'} \neq e_{k+1}$, then \bloo{$e_{k'} \in \sigma(k+2, k+t-2)$}. Furthermore, since $n\geq s+2t-2$, the sets $\sigma(i, i+s-1)$ and $\sigma(k+1, k+t)$ do not intersect. However, $e_k \notin \sigma(k+1, k+t)$, contradicting \Cref{consec_nearby}. Thus either $e_{k'}=e_{k-1}$ or $e_{k'}=e_{k+1}$, thereby completing the proof of the lemma.
\end{proof}

\begin{lemma}
\label{2apart_neq}
If $n \geq s+2t-1$ and \blue{$n \geq 2s+t-4$}, then $c_i\neq c_{i+2}$ for all $i \in [n]$.
\end{lemma}


\begin{proof}
Suppose $c_i = c_{i+2}$ for some $i\in [n]$. Then $C_i = \sigma(i, i+s-2) \cup \{c_i\}$ and $C_{i+2} = \sigma(i+2, i+s) \cup \{c_i\}$. By circuit elimination, there is also a circuit, say $C$, of $M$ contained in $\sigma(i, i+s)$. If $C$ contains neither $e_{i+s-1}$ nor $e_{i+s}$, then $C$ is contained in $\sigma(i, i+s-2)$, and thus properly contained in $C_i$, a contradiction. So $C$ contains at least one of $e_{i+s-1}$ and $e_{i+s}$. We next show that $c_i$ is contained in $\sigma(i+s+1, i+s+t-1)$.

First, if $e_{i+s}$ is not contained in $C$, then \bloo{$e_{i+s-1} \in C$}, in which case the $(t-1)$-element set $\sigma(i+s-1, i+s+t-3)$ intersects $C$ in one element. Therefore, by orthogonality, \bloo{$c^*_{i+s-1} \in \sigma(i,i+s-2)$}. Now, orthogonality between $C_i$ and $C^*_{i+s-1}$ implies \bloo{$c_i \in \sigma(i+s-1, i+s+t-3)$}. Furthermore, $c_i$ can be neither $e_{i+s-1}$ nor $e_{i+s}$ since these elements are contained in $\sigma\left(i+2, i+s\right)$ and $c_i=c_{i+2}$, so \bloo{$c_i \in \sigma(i+s+1, i+s+t-3)$}.

Now assume that $e_{i+s} \in C$. Orthogonality with $C^*_{i+s}$ implies that \bloo{$\left.c^*_{i+s} \in \sigma(i,i+s-1)\right.$}, so either $c^*_{i+s} = e_{i+s-1}$ or \bloo{$\left.c^*_{i+s} \in \sigma(i, i+s-2)\right.$}. In the latter case, orthogonality with $C_i$ implies that \bloo{$\left.c_i \in \sigma(i+s+1,i+s+t-2)\right.$}. Thus, we may assume that $c^*_{i+s} = e_{i+s-1}$. Now, $C^*_{i+s}$ intersects $\sigma(i+1,i+s-1)$ in one element, so \bloo{$c_{i+1} \in \sigma(i+s,i+s+t-2)$}. Either $c_{i+1} = e_{i+s}$, or \bloo{$c_{i+1} \in \sigma(i+s+1, i+s+t-2)$}. Say $c_{i+1} = e_{i+s}$. Then both $\sigma(i+1, i+s)$ and $\sigma(i+2, i+s)\cup \{c_i\}$ are circuits of $M$ (noting that $c_i \neq e_{i+1}$ because \bloo{$\left.e_{i+1} \in \sigma(i, i+s-2)\right.$}). This contradicts \Cref{circ_unique}, so \bloo{$c_{i+1} \in \sigma(i+s+1,i+s+t-2)$}. Since \bloo{$c_{i+1}\notin \sigma(i-1, i+s)$}, and $n\ge s+2t-1$ and $n\ge 2s+t-4$, it follows by \Cref{consecutive} that the elements $c_i$ and $c_{i+1}$ are consecutive, so $c_i \in \sigma(i+s+1, i+s+t-1)$.

We have now shown that, in all cases, \bloo{$c_i \in \sigma(i+s+1, i+s+t-1)$}. But, using a symmetrical argument and comparing $C$ and $C_{i+2}$, we can show that \bloo{$c_{i+2} \in \sigma(i-t+1, i-1)$}. Now, $c_{i+2} = c_i$, so \bloo{$c_i \in \sigma(i-t+1, i-1)$} and \bloo{$c_i \in \sigma(i+s+1, i+s+t-1)$}. But, since $n \geq s+2t-1$, these two sets are disjoint. This contradiction completes the proof of the lemma.
\end{proof}

\begin{lemma}
\label{gap0}
Let $n \geq s+2t-1$ and $t \geq s$. If there exists $i \in [n]$ such that $\sigma(i, i+s-1)$ is a circuit of $M$, then $M$ is $(s, t)$-cyclic.
\end{lemma}

\begin{proof}
Let $i\in [n]$ such that $\sigma(i, i+s-1)$ is a circuit of $M$. We will show that $\sigma(i+2, i+s+1)$ is also a circuit. It then follows that $\sigma(i+2k, i+2k+s-1)$ is a circuit for all $k\geq 1$ and so, by \Cref{total_circuits}, $M$ is $(s, t)$-cyclic.

Since $\sigma(i, i+s-1)$ is a circuit, it follows by \Cref{circ_unique} that $c_i = e_{i+s-1}$ and $c_{i+1} = e_i$. By \Cref{consecutive}, either $c_{i+2} \in \sigma(i, i+s+1)$ or $c_{i+2} = e_{i-1}$ or $c_{i+2} = e_{i+1}$. Therefore, $c_{i+2} \in \{e_{i-1} ,e_i, e_{i+1}, e_{i+s+1}\}$. If $c_{i+2} = e_{i+s+1}$, then $\sigma(i+2, i+s+1)$ is a circuit, and we have the desired result. 
Furthermore, if $c_{i+2} = e_i$, then $c_{i+2} = c_{i+1}$, contradicting \Cref{consec_neq}. If $c_{i+2} = e_{i+1}$, then both $\sigma(i, i+s-1)$ and $\sigma(i+1, i+s)$ are circuits containing $\sigma(i+1, i+s-1)$, contradicting \Cref{circ_unique}. Therefore we may assume that $c_{i+2} = e_{i-1}$.

Now consider $c_{i+3}$. Since $c_{i+2}$ is not contained in $\sigma(i+1,i+s+2)$, it follows by \Cref{consecutive} that either $c_{i+3} = e_{i-2}$ or $c_{i+3} = e_i$. But $c_{i+1} = e_i$, so $c_{i+3} \neq e_i$ by \Cref{2apart_neq}. Therefore, $c_{i+3} = e_{i-2}$. More generally, suppose that $c_{i+k-2} = e_{i-k+3}$ and $c_{i+k-1} = e_{i-k+2}$, for some $k \geq 4$. If $n \geq 2k+s-2$, then $c_{i+k-1} \notin \sigma(i+k-2,i+k+s-1)$, and we can apply \Cref{consecutive} to show that $c_{i+k} \in \{e_{i-k+1},e_{i-k+3}\}$. But $c_{i+k-2} = e_{i-k+3}$, so $c_{i+k} = e_{i-k+1}$ \blue{by \Cref{2apart_neq}}.

\blue{By induction, we deduce}, for all $k \geq 2$ satisfying $n \geq 2k+s-2$, that $c_{i+k} = e_{i-k+1}$. \bloo{Suppose $t=s$. Taking $k=s$, we have that $n \geq 3s-2$, and so $c_{i+s} = e_{i-s+1}$. Therefore, assuming $t > s$, we have that $c_{i+s} = e_{i-s+1} \in \sigma(i-t+2,i)$.} This means that the $(t-1)$-element set $\sigma(i-t+2,i)$ intersects each of $C_i$ and $C_{i+s}$ in one element, and so $c^*_{i-t+2} \in C_i \cap C_{i+s}$. But $C_i$ and $C_{i+s}$ are disjoint, a contradiction. Thus, we may assume that $s=t$.

\blue{\bloo{We apply \Cref{consecutive} to $c_{i-1}$ with the aim of showing that $c_{i-1}=e_{i+s}$}. Suppose $c_{i-1} = e_j$. If $c_{i-1} \notin \sigma(i-2,i+s-1)$, then either $c_i = e_{j-1}$ or $c_i = e_{j+1}$. Since $c_i = e_{i+s-1}$, it follows that either $c_{i-1} \in \sigma(i-2,i+s-1)$ or $c_{i-1} = e_{i+s}$.} Now consider the $(t-1)$-element set $\sigma(i+s, i+s+t-2)$. This intersects $C_{i+2} = \sigma(i+2, i+s) \cup \{e_{i-1}\}$ in one element. So, either $c^*_{i+s} \in \sigma(i+2,i+s-1)$ or $c^*_{i+s} = e_{i-1}$. In the former case, $C^*_{i+s}$ intersects $\sigma(i, i+s-1)$ in one element, contradicting orthogonality. So $c^*_{i+s} = e_{i-1}$. But then $\sigma(i-1, i+s-3)$ intersects $C^*_{i+s}$ in one element, and so $c_{i-1} \in \sigma(i+s, i+s+t-2)$. Therefore, \bloo{$c_{i-1} \notin \sigma(i-2,i+s-1)$}, and so $c_{i-1} = e_{i+s}$.

Consider $c_{i-2}$. Since \bloo{$c_{i-1} \notin \sigma(i-3, i+s-2)$}, it follows by \Cref{consecutive} that either $c_{i-2}=e_{i+s-1}$ or $c_{i-2}=e_{i+s+1}$. But $c_i = e_{i+s-1}$ and so, by \Cref{2apart_neq}, $c_{i-2} = e_{i+s+1}$. More generally, suppose $c_{i-k+3} = e_{i+s+k-4}$ and $c_{i-k+2} = e_{i+s+k-3}$, for some $k \geq 4$. If $n \geq 2k+s-2$, then $c_{i-k+2} \notin \sigma(i-k, i-k+s+1)$, and we can apply Lemma~\ref{consecutive} to show that $c_{i-k+1} \in \{e_{i+s+k-4}, e_{i+s+k-2}\}$. But $c_{i-k+3} = e_{i+s+k-4}$, so $c_{i-k+1} = e_{i+s+k-2}$.

Therefore, by induction, for all $k \geq 2$ satisfying $n \geq 2k+s-2$, we have $c_{i+k} = e_{i-k+1}$ and $c_{i-k+1}=e_{i+s+k-2}$. If $s=t=3$, we have $c_{i+2}=e_{i-1}$ and $c_{i-1}=e_{i+3}$. By orthogonality between $C^*_i$ and $C_{i-1}$, we have \bloo{that either} $c^*_i=e_{i-1}$ or $c^*_i=e_{i+3}$. For either possibility, $C^*_i$ intersects $C_{i+2}$ in one element, a contradiction. Now assume that $s=t\ge 4$, and consider the $(t-1)$-element set $\sigma(i,i+t-2)$. This set intersects each of $\sigma(i-s+2, i)$ and $\sigma(i+t-2, i+s+t-4)$ in exactly one element. Now, since $n \geq 3s-4$, we have that $c_{i-s+2} = e_{i+2s-3}$ and, since $n \geq s+2t-6$, we have that $c_{i+t-2}=e_{i-t+3}=e_{i-s+3}$. Neither $c_{i-s+2}$ nor $c_{i+t-2}$ are contained in $\sigma(i,i+t-2)$, and so $c^*_{i} \in C_{i-s+2} \cap C_{i+t-2} = \{e_{i-s+3}\}$. But now, \blue{since $c_{i-s+1} = e_{i+2s-2}$, we have that} $C_{i-s+1} = \sigma(i-s+1,i-1) \cup \{e_{i+2s-2}\}$, \bloo{which} intersects $C^*_i$ in one element. This contradiction to orthogonality completes the proof of the lemma.
\end{proof}

\begin{lemma}
\label{gap1}
Let $n \geq s + 2t - 1$, and suppose that $t\ge s$. If \bloo{$c_{i}=e_{i+s}$, then $c_{i+1}=e_{i+s+1}$}.
\end{lemma}

\begin{proof}
As $t\ge s$, it follows by \Cref{consecutive} that either $c_{i+1} \in \sigma(i-1,i+s)$, or $c_{i+1} = e_{i+s+1}$. Therefore, $c_{i+1} \in \{e_{i-1},e_i,e_{i+s},e_{i+s+1}\}$. By \Cref{consec_neq}, $c_{i+1} \neq e_{i+s}$. Also, if $c_{i+1} = e_i$, then both $\sigma(i, i+s-1)$ and $\sigma(i, i+s-2) \cup \{e_{i+s}\}$ are circuits containing $\sigma(i, i+s-2)$, contradicting \Cref{circ_unique}.

Suppose $c_{i+1} = e_{i-1}$, and consider the $(t-1)$-element set $\sigma(i-t+1, i-1)$. As $n\ge s+2t-1$, this set intersects $C_{i+1}$ in exactly one element, but does not intersect $C_i$. Therefore, \bloo{$c^*_{i-t+1} \in C_{i+1}$}, but not in \bloo{$c^*_{i-t+1} \notin C_i$}; the only possibility is $c^*_{i-t+1} = e_{i+s-1}$.

Now consider the $(t-1)$-element set $\sigma(i+s, i+s+t-2)$. As $n\ge s+2t-1$, this set intersects $C_i$ in exactly one element, and does not intersect $C_{i+1}$. Therefore, $c^*_{i+s}=e_i$. Finally, consider the $(s-1)$-element set $\sigma(i+2, i+s)$. This last set intersects each of $C^*_{i+s}$ and $C^*_{i-t+1}$ in exactly one element. But $C^*_{i+s}$ and $C^*_{i-t+1}$ are disjoint, a contradiction. Therefore, $c_{i+1} = e_{i+s+1}$.
\end{proof}

\begin{lemma}
\label{allgap}
Let $n \geq s+2t-1$ and $t \geq s$. \bloo{If $c_i=e_{i+s-1+k}$ for some $1 \leq k < n-s$, then $c_{i+1} = e_{i+s+k}$}.
\end{lemma}

\begin{proof}
The proof is by induction on $k$. If $k=1$, then the result follows immediately from \Cref{gap1}. Suppose $k = 2$, \blue{so that, $c_i = e_{i+s+1}$}. By \Cref{consecutive}, either $c_{i+1} = e_{i+s}$ or $c_{i+1} = e_{i+s+2}$. If $c_{i+1} = e_{i+s}$, then $\sigma(i+1,i+s)$ is a circuit. But, by \Cref{gap0}, this implies $M$ is $(s, t)$-cyclic, which, by \Cref{circ_unique}, contradicts the uniqueness of the circuit containing $\sigma(i, i+s-2)$. So $c_{i+1} = e_{i+s+2}$, and the lemma holds for $k=2$.

Now let $k\ge 3$, and suppose that, \bloo{for all $i' \in [n]$, if $c_{i'} = e_{i'+s-1+(k-2)}$, then $c_{i'+1} = e_{i'+s+(k-2)}$.} We shall complete the proof by proving that the lemma holds for $k$. So, let $c_i=e_{i+s-1+k}$. Then, by \Cref{consecutive}, either $c_{i+1} = e_{i+s-2+k}$ or $c_{i+1} = e_{i+s+k}$. If $c_{i+1}=e_{i+s-2+k}$, then, by the induction assumption, $c_{i+2} = e_{i+s-1+k}$. But now $c_{i+2} = c_i$. This contradiction to \Cref{2apart_neq} shows that $c_{i+1} = e_{i+s+k}$, and completes the proof of the lemma.
\end{proof}

At last we are ready to prove \Cref{par_to_tot}.

\begin{proof}[Proof of \Cref{par_to_tot}]
Since $\sigma$ is an $(s, t)$-cyclic ordering of $M$ if and only if $\sigma$ is a $(t, s)$-cyclic ordering of $M^*$, we may assume, without loss of generality, that $t \geq s$. For the purposes of obtaining a contradiction, suppose there is no $j \in [n]$ such that $\sigma(j, j+s-1)$ is a circuit of $M$. Since $\sigma$ is a nearly $(s, t)$-cyclic ordering of $M$, it follows by \Cref{allgap} that there exists \blue{$1 \leq k < n-s$} such that, for all $i \in [n]$, the set $\sigma(i,i+s-2) \cup \{e_{i+s-1+k}\}$ is a circuit. In particular, by \Cref{circ_unique}, $C_i=\sigma(i, i+s-2)\cup \{e_{i+s-1+k}\}$. Take one such $i$, and consider the $(t-1)$-element set $\sigma(i, i+t-2)$. As $n\ge 2s+t-3$, the $(s-1)$-element sets $\sigma(i-s+1, i-1)$ and $\sigma(i+t-1, i+s+t-3)$ are disjoint, so at least one of these two sets does not contain $c^*_i$.
We will establish a contradiction for when $c^*_i\not\in \sigma(i-s+1, i-1)$. A symmetrical argument applies when $c^*_i\not\in \sigma(i+t-1, i+s+t-3)$. \blue{So suppose} $c^*_i \notin \sigma(i-s+1,i-1)$. Then $\sigma(i-s+2, i)$ intersects $C^*_i$ in exactly one element. Therefore, either $c_{i-s+2}=c^*_i$ or $c_{i-s+2} \in \sigma(i+1, i+t-2)$.

First assume that $c_{i-s+2} \in \sigma(i+1, i+t-2)$. We know that $c_{i-s+2} \neq e_{i+1}$, \blue{for otherwise} $\sigma(i-s+2, i+1)$ is a circuit. So $c_{i-s+2} \in \sigma(i+2,i+t-2)$. But now, by \Cref{allgap}, $c_{i-s+1} \in \sigma(i+1, i+t-3)$, and so $C_{i-s+1}$ and $C^*_i$ intersect in exactly one element, a contradiction.

Now assume that $c_{i-s+2} = c^*_i$. Consider the $(s-1)$-element set $\sigma\left(i+t-2, i+s+t-4\right)$. Suppose \bloo{$\left.c^*_i \notin \sigma(i+t-1, i+s+t-3)\right.$}. Then, by orthogonality, either $c_{i+t-2} = c^*_i$ or $c_{i+t-2} \in \sigma(i,i+t-3)$. But $c_{i+t-2} \neq e_{i+t-3}$, since then $\sigma(i+t-3, i+s+t-4)$ is a circuit, and $c_{i+t-2} \notin \sigma(i, i+t-4)$ since then $C_{i+t-1}$ and $C^*_i$ intersect in exactly one element, \blue{by \Cref{allgap}}. Furthermore, $c_{i+t-2} \neq c^*_i$, since then $c_{i+t-2} = c_{i-s+2}$, contradicting Lemmas~\ref{circ_unique} and~\ref{allgap}. Therefore, $c^*_i \in \sigma(i+t-1, i+s+t-3)$.

It now follows that \bloo{$c_{i-s+2} = e_{i+t-2+\ell}$ for some $1 \leq \ell \leq s-1$}. Therefore, by \Cref{allgap}, $c_{i-s+2-\ell} = e_{i+t-2}$. Furthermore, as $n \geq 3s+t-5$, the $(s-1)$-element set $\sigma(i-s+2-\ell, i-\ell)$ does not contain $c^*_i = e_{i+t-2+\ell}$ and does not intersect $\sigma(i, i+t-2)$. So $C_{i-s+2-\ell}$ and $C^*_i$ intersect in exactly one element. This contradiction to orthogonality establishes that $M$ must contain a circuit $\sigma(j, j+s-1)$ for some $j\in [n]$, and so, by \Cref{gap0}, $\sigma$ is an $(s, t)$-cyclic ordering of $M$. This completes the proof of the theorem.
\end{proof}

\section{Proof of Theorem~\ref{free}}
\label{freethm}

In this section, we prove Theorem~\ref{free}. We begin by defining a class of matroids that contains, for all positive integers $s$ exceeding one and all positive even integers $n$, the matroid $\Psi^n_s$. The proof of Theorem~\ref{free} is a consequence of a more general weak-map result, namely Theorem~\ref{weak_map}, that we establish for this class.

\blue{Recall that for} a vertex $v$ of a graph $G$, we denote the set of vertices of $G$ adjacent to $v$, that is, the {\em neighbours of $v$}, by $N(v)$. More generally, for a subset $U$ of vertices of $G$, the {\em neighbours of $U$}, denoted $N(U)$, is
$$\bigcup_{v\in U} N(v).$$

\bloo{We next define a multi-path matroid. Let $E$ be a set of $n$ elements, and suppose that $\sigma=(e_1,e_2,\ldots,e_n)$ is a cyclic ordering of $E$. Let $m$ be a positive integer \bloo{exceeding one}. Choose distinct elements $x_1, x_2, \ldots, x_m \in [n]$ and distinct elements $y_1,y_2,\ldots,y_m \in [n]$ such that $e_{x_i} \in \sigma(x_{i-1},x_{i+1})$ and $e_{y_i} \in \sigma(y_{i-1},y_{i+1})$ for all $i \in [m]$, where subscripts  \bloo{of $x$ and $y$} are interpreted modulo $m$, and, furthermore, the intervals $\sigma(x_i,y_i)$ form an anti-chain of $\sigma$, that is, there is no $i, i' \in [m]$ such that $\sigma(x_i,y_i) \subseteq \sigma(x_{i'},y_{i'})$.} Let $G$ denote the bipartite graph with parts $E$ and $[m]$, and whose set of edges satisfy $\bloo{N(i)} = \sigma(x_i, y_i)$ for all $i \in [m]$. The transversal matroid on ground set $E$ with presentation
$$\mathscr I=(N(1), N(2), \ldots, N(m))$$
is called a {\em multi-path matroid} and is denoted by $M[\mathscr I]$. \blue{Let $M^*[\mathscr I]$ denote the dual of $M[\mathscr I]$, and observe that, for all $i\in [m]$, the set $\sigma(x_i, y_i)$ is a circuit of $M^*[\mathscr I]$. Multi-path matroids were introduced in~\cite{bon07}.}

\blue{As an example, let $s$ be a positive integer exceeding one and let $n$ be a positive even integer, and suppose that $\sigma=(e_1, e_2, \ldots, e_n)$ is a cyclic ordering of $E$ and $m=\frac{n}{2}$. By choosing \bloo{$x_i=2i-1$ and $y_i=2i+s-2$} for all $i\in \left[\frac{n}{2}\right]$, we have that $G\cong G^n_s$, the bipartite graph defined in the introduction, and $M^*[\mathscr I] \cong \Psi^n_s$.}

The initial goal of this section is to establish \Cref{weak_map} which says that, up to isomorphism, $M^*[\mathscr I]$ is \blue{at least as free as} any matroid on the same ground set satisfying a certain rank condition; that is, up to isomorphism, every such matroid is a weak-map image of $M^*[\mathscr I]$.


A subset $X \subseteq E$ is independent in $M^*[\mathscr I]$ if and only if $E-X$ is cospanning. In other words, $X$ is independent in $M^*[\mathscr I]$ if and only if there is a complete matching from $[m]$ into $E-X$. By Hall's Theorem~\cite{hal35}, this is true precisely if, for all subsets $J$ of $[m]$, we have that $|N(J)-X| \geq |J|$. We repeatedly use this fact in the proofs in this section. To ease reading, in the statements of these lemmas and theorem, the multi-path matroid $M[\mathscr I]$ has ground set $E$ and is constructed as above.

\begin{lemma} \label{psi_rank}
\bloo{$r(M^*[\mathscr I]) = |E|-m$.}
\end{lemma}

\begin{proof}
\bloo{It is sufficient to prove that $r(M[\mathscr I]) = m$. Let $X \subseteq E$ be a set of $m+1$ elements. Clearly there is no matching of $X$ into $[m]$, so $X$ is dependent. Therefore, $r(M[\mathscr I]) \leq m$. For all $i \in [m]$, we have that $\{i,e_{x_i}\}$ is an edge of the bipartite graph $G$. Therefore, $\{\{1,e_{x_1}\},\{2,e_{x_2}\},\ldots,\{m,e_{x_m}\}\}$ is a matching of $G$. Hence $r(M[\mathscr I]) \geq m$, so $r(M[\mathscr I]) = m$, completing the proof.}
\end{proof}

\begin{lemma}
\label{circuits1}
Let $C$ be a circuit of $M^*[\mathscr I]$. Let $J \subseteq [m]$ such that $\left.|N(J)-C| < |J|\right.$. Then $C$ is a subset of $N(J)$ containing $|N(J)|-|J|+1$ elements.
\end{lemma}

\begin{proof}
If $C$ is not a subset of $N(J)$, then there exists an element $e$ of $C$ such that $e\not\in N(J)$. Then
$$|N(J)-(C-\{e\})| = |N(J)-C| < |J|.$$
But this implies that $C-\{e\}$ is dependent, a contradiction. Thus $C$ is a subset of $N(J)$.

To see that $C$ contains $|N(J)|-|J|+1$ elements, suppose that $\left.|N(J)-C| < |J|-1\right.$, and let $e \in C$. Then, as $C$ is a subset of $N(J)$, we have
$$|N(J)-(C-\{e\})| = |N(J)-C|+1 < |J|.$$
Again, this implies $C-\{e\}$ is dependent, a contradiction. Thus
$$|N(J)|-|C| = |N(J)-C| = |J|-1.$$
Rearranging this last equation gives $|C| = |N(J)|-|J|+1$, thereby completing the proof of the lemma.
\end{proof}


\begin{lemma}
\label{circuits}
Let $C$ be a circuit of $M^*[\mathscr I]$. Then either $C$ has $|E|-m+1$ elements or there \bloo{exist} $i, j\in [m]$ such that each of the following hold:
\bloo{\begin{enumerate}[{\rm (i)}]
\item $N([i, j]) = \sigma(x_i, y_j)$,
\item $C$ is a subset of $N([i, j])$ containing $|N([i, j])|-|[i,j]|+1$ elements,
\item either $i = j$, or $N([i, j])-N([i+1, j]) \subseteq C$,
\item either $i = j$, or $N([i, j])-N([i, j-1]) \subseteq C$, and
\item $\sigma(x_i, y_j) \subseteq \cl(C)$,
\end{enumerate}}
\end{lemma}

\begin{proof}
Since $C$ is dependent, there exists $J \subseteq [m]$ such that $|N(J)-C| < |J|$. If $N(J)=E$, then $N([m])=E$, so $|N([m])-C| = |E-C| < |J| \leq m$. Therefore, by \Cref{circuits1}, $C$ has $|E|-m+1$ elements. So suppose that $N(J) \neq E$.

We next show that we may assume that $J$ has the property that $\left.N(J)=\sigma(x_i,y_j)\right.$ for some $i,j \in [m]$. If $J$ does not satisfy this property, then partition $J$ into maximal subsets with disjoint, consecutive neighbourhoods. \bloo{More formally, since 
\[N(J) = \bigcup\limits_{i_0 \in J}\sigma(x_{i_0},y_{i_0}),\] 
we may partition $J$ into sets $J_1, J_2, \ldots, J_k$ such that, for all $\ell \in [k]$, there \bloo{exist} $i_\ell, j_\ell \in [m]$ with $N(J_\ell) = \sigma(x_{i_\ell}, y_{j_\ell})$. Furthermore, we may choose such a partition in which, for all distinct $\ell,\ell' \in [k]$, the sets $\sigma(x_{i_{\ell}},y_{j_\ell})$ and $\sigma(x_{i_\ell'},y_{j_\ell'})$ are disjoint.} Now,
\begin{align*}
|N(J_1)-C|+|N(J_2)-C|+\cdots+|N(J_k)-C| & = |N(J)-C| \\
&< |J| \\
&= |J_1| + |J_2| + \cdots + |J_k|.
\end{align*}
It follows that there exists $\ell \in [k]$ such that $|N(J_\ell)-C| < |J_\ell|$, in which case replace $J$ with $J_{\ell}$.

We have chosen $J \subseteq [m]$ such that $|N(J)-C| < |J|$ and $N(J) = \sigma(x_i,y_j)$ for some $i,j \in [m]$. It follows from the definition of the bipartite graph $G$ that $J \subseteq [i,j]$. Furthermore, $N([i,j]) \subseteq \sigma(x_i,y_j)$, so $N([i,j]) = \sigma(x_i,y_j)$, that is (i) holds. Therefore,
\[|N([i,j])-C|=|N(J)-C| < |J| \leq \left|[i,j]\right|.\]
Hence, by \Cref{circuits1}, $C$ is a subset of $N([i,j])$ containing $\left.|N([i, j])|-|[i, j]|+1\right.$ elements, so (ii) holds.

\bloo{We next show that we may choose $i' \in [m]$ such that the pair $i',j$ satisfies (i), (ii), and (iii). Initially, choose $i'=i$, and suppose $i'$ and $j$ do not satisfy (iii). Then $i' \neq j$, and there exists $f \in N([i', j]) - N([i'+1, j])$ with $f \notin C$. First, assume $\left.N([i',j])-N([i'+1,j]) = \{f\}\right.$. Then $C$ is a subset of $N([i'+1,j])$ and
\begin{align*}
|C| &= |N([i',j])| - |[i',j]| + 1 \\
&= (|N([i'+1,j])|+1) - (|[i'+1,j]|+1) + 1 \\ 
&= |N([i'+1,j])| - |[i'+1,j]| + 1,
\end{align*}
so $i'+1,j$ satisfies (ii). Furthermore, it follows from the definition of the bipartite graph $G$ that, since $N([i',j])=\sigma(x_{i'},y_j)$, we have that $\left.N([i'+1,j])=\sigma(x_{i'+1},y_j)\right.$. Thus, $i'+1,j$ satisfies (i) and (ii), and we may replace $i'$ in the pair $i',j$ with $i'+1$.}

Hence, we may assume there exists $f' \in N([i',j])-N([i'+1,j])$ with $f' \neq f$. First assume $f' \in C$. Then, by~(ii),
\begin{align*}
|N([i'+1, j])-(C-\{f'\})| & = |N([i'+1, j])-C| \\
& < |N([i', j])-C| \\
& = \left|[i', j]\right|-1 \\
& = \left|[i'+1, j]\right|.
\end{align*}
Therefore, $C-\{f'\}$ is dependent, a contradiction. Now assume $f'\notin C$. Since $f, f'\not\in C$,
\[|N([i'+1, j])-C| < |N([i', j])-C|-1.\]
Let $x\in C$. Then, by~(ii),
\begin{align*}
|N([i'+1, j])-(C-\{x\})| & \leq |N([i'+1, j])-C|+1 \\
& < |N([i', j])-C| = \left|[i', j]\right|-1 = \left|[i'+1, j]\right|.
\end{align*}
But this implies that $C-\{x\}$ is dependent, and thus the pair $i',j$ satisfies (i), (ii) and (iii). \bloo{A symmetrical argument shows that we may choose $j' \in [m]$ such that the pair $i',j'$ satisfies (i)-(iv).}

\bloo{It remains to show (v).} Let $e \in C$, and let $e' \in \sigma(x_{i'},y_{j'})-C$. Then
\[
|N([i',j'])-((C-\{e\})\cup\{e'\})| = |N([i',j'])-C| < |[i',j']|.
\]
Therefore, $(C-\{e\})\cup\{e'\}$ is dependent, so contains a circuit $C'$. The circuit $C'$ contains the element $e'$, as otherwise $C'$ is a proper subset of $C$. Therefore, $e' \in \cl(C)$, completing the proof of the lemma.
\end{proof}

\begin{theorem}
\label{weak_map}
\blue{Let $M$ be a matroid on ground set $E$ such that, for all $i \in [m]$ and $1 \leq k \leq m$, we have
\begin{align*}
r_M\big(\sigma(x_i, y_i)\cup & \sigma(x_{i+1}, y_{i+1})\cup \cdots\cup \sigma(x_{i+k-1}, y_{i+k-1})\big) \\
& \leq r_{M^*[\mathscr I]}\big(\sigma(x_i, y_i)\cup (x_{i+1}, y_{i+1})\cup \cdots\cup \sigma(x_{i+k-1}, y_{i+k-1})\big).
\end{align*}
If $M[\mathscr I]$ has no \bloo{loops}}, then, under the identity map, $M$ is a weak-map image of $M^*[\mathscr I]$.
\end{theorem}

\begin{proof}
Let $\varphi$ denote the identity map from the ground set $E$ of $M^*[\mathscr I]$ to the ground set $E$ of $M$. To prove the theorem, we will show that if $C$ is a circuit of $M^*[\mathscr I]$, then $\varphi(C)$ contains a circuit of $M$. Let $C$ be a circuit of $M^*[\mathscr I]$. Now, as $M[\mathscr I]$ has no \bloo{loops, every element of $E$ is in $N(i) = \sigma(x_i,y_i)$ for some $i \in [m]$. Therefore, $\sigma(x_1,y_1) \cup \sigma(x_2,y_2) \cup \cdots \cup \sigma(x_m,y_m) = E$. Thus, by \Cref{psi_rank}}
\begin{align*}
|E|-m & = r(M^*[\mathscr I]) \\
& = r_{M^*[\mathscr I]}\big(\sigma(x_1, y_1)\cup \sigma(x_2, y_2)\cup \cdots\cup \sigma(x_m, y_m)\big) \\
& \ge r_M\big(\sigma(x_1, y_1)\cup \sigma(x_2, y_2)\cup \cdots\cup \sigma(x_m, y_m)\big) \\
& = r(M).
\end{align*}
\blue{Therefore, if $C$ contains $|E| - m + 1$ elements, then $\varphi(C)$ contains a circuit of $M$.}


Otherwise, by \Cref{circuits}, there \bloo{exist} $i,j \in [m]$ such that $C$ is a subset of $\sigma(x_i,y_j)$ containing $|\sigma(x_i,y_j)|-|[i,j]|+1$ elements. \bloo{Furthermore, by \Cref{circuits}(i), we have that
\[
N([i,j]) = \sigma(x_i,y_i) \cup \sigma(x_{i+1},y_{i+1}) \cup \cdots \cup \sigma(x_j,y_j) = \sigma(x_i,y_j)
\]
and so $r_M(\sigma(x_i,y_j)) \leq r_{M^*[\mathscr I]}(\sigma(x_i,y_j))$. By \Cref{circuits}(v), we have that $\sigma(x_i,y_j) \subseteq \cl(C)$, so $r_{M^*[\mathscr I]}(\sigma(x_i,y_j)) = r_{M^*[\mathscr I]}(C) = |C|-1$.} \bloo{Thus, 
\[r_M(C) \leq r_M(\sigma(x_i,y_j)) \leq r_{M^*[\mathscr I]}(\sigma(x_i,y_j))=|C|-1.\]
Therefore, $\varphi(C)$ contains a circuit of $M$.}
\end{proof}


\blue{The previous results in this section apply for any multi-path matroid $M^*[\mathscr I]$. We now turn our attention to the case where $M^*[\mathscr I]\cong \Psi^n_s$, towards proving \Cref{free}}. \bloo{We first show that $\Psi^n_s$ is self-dual.}

\begin{lemma} \label{selfdual}
	\bloo{Let $s$ be an integer exceeding two, and let $\phi_s:E \rightarrow E$ be the identity map if $s$ is even, or the map $\phi_s(e_i) = e_{i+1}$ if $s$ is odd. Then $\Psi^n_s$ is self-dual under the map $\phi_s$.}
\end{lemma}

\begin{proof}
\bloo{By \Cref{circuits}, a circuit of $\Psi^n_s$ is either a set of $\frac{n}{2}+1$ elements, or a subset of $\sigma(x_i,y_{i+k})=\sigma(2i-1,2i+2k+s-2)$ containing $\left.|\sigma(2i-1,2i+2k+s-2)|-(k+1)+1 = s+k\right.$ elements, for some $i \in [m]$ and $k \leq \frac{n}{2}-s$. It follows from \Cref{psi_rank} that a set $X$ is a basis of $\Psi^n_s$ if and only if $X$ is a set of $\frac{n}{2}$ elements such that, for all odd $i \in [n]$ and $k \leq \frac{n}{2}-s$, we have that
\begin{align*}
\left|X \cap \sigma(i,i+s-1+2k)\right| < s+k.
\end{align*}
Let $B$ be a basis of $\Psi^n_s$. Let $i \in [n]$ be odd, and let $k \leq \frac{n}{2}-s$. To complete the proof, it is sufficient to show that $\left|\phi_s^{-1}(E-B) \cap \sigma(i,i+s-1+2k)\right| < s+k$.}
	
\bloo{First, suppose $s$ is even. Then 
\begin{align*}
\phi_s(E-\sigma(i,i+s-1&+2k)) = \sigma(i+s+2k,i-1) \\
&= \sigma\left(i+s+2k,i+s+2k+s-1+2\left(\tfrac{n}{2}-k-s\right)\right).
\end{align*}
Hence, noting that $i+s+2k$ is odd, we have that 
\[\left|B \cap \phi_s\left(E-\sigma(i,i+s-1+2k)\right)\right| < \tfrac{n}{2}-k.\]
On the other hand, if $s$ is odd, then
\begin{align*}
\phi_s(E-\sigma(i,i&+s-1+2k)) = \phi_s(\sigma(i+s+2k,i-1)) \\
&= \sigma(i+s+2k+1,i)\\ 
&= \sigma\left(i+s+2k+1, i+s+2k+1 + s-1 + 2\left(\tfrac{n}{2}-k-s\right)\right).
\end{align*}
Since $i+s+2k+1$ is odd, we have that 
\[
\left|B \cap \phi_s\left(E-\sigma(i,i+s-1+2k)\right)\right| < \tfrac{n}{2}-k.
\]}
\bloo{In both cases, 
\[
\left|\phi_s^{-1}(B) \cap \left(E-\sigma(i,i+s-1+2k)\right)\right| < \tfrac{n}{2}-k
\]
and so 
\[
\left|\phi_s^{-1}(B) \cap \sigma(i,i+s-1+2k)\right| > k.
\]
Therefore, 
\begin{align*}
\left|\phi_s^{-1}(E-B) \cap \sigma(i,i+s-1+2k)\right| &< |\sigma(i,i+s-1+2k)| - k \\
&= s+2k-k = s+k
\end{align*}
as required.}
\end{proof}

\begin{lemma}
Let $s$ and $t$ be positive integers exceeding one, such that $t \ge s$. If $n$ is a positive even integer \blue{with $n \geq s+t-2$} and $s\equiv t\bmod{2}$, then $T^{\frac{t-s}{2}}(\Psi^n_s)$ is an $(s, t)$-cyclic matroid.
\label{truncation2}
\end{lemma}

\begin{proof}
Without loss of generality, we may assume that the ground set $\{e_1, e_2, \ldots, e_n\}$ of $\Psi^n_s$ is consistent with the bipartite graph $G^n_s$ associated with the dual of $\Psi^n_s$ as described in the introduction. In particular, $G^n_s$ has vertex parts $\{e_1, e_2, \ldots, e_n\}$ and $[\frac{n}{2}]$ and, for all $i\in \{1, 2, \ldots, \frac{n}{2}\}$, we have
$$\bloo{N(i)}=\{e_{2i-1}, e_{2i}, \ldots, e_{2i+s-2}\}.$$

The proof is by induction on $t$. Suppose that $t=s$, and consider $\left.T^0(\Psi^n_s)=\Psi^n_s\right.$. It is easily checked that, for all odd $i\in [n]$, the set $\{e_i, e_{i+1}, \ldots, e_{i+s-1}\}$ is an $s$-element circuit of $\Psi^n_s$. \bloo{By \Cref{selfdual}, the set $\{e_j,e_{j+1},\ldots,e_{j+s-1}\}$ is an $s$-element cocircuit of $\Psi^n_s$ for all odd $j \in [n]$ if $s$ is even, or for all even $j \in [n]$ if $s$ is odd. Therefore $\Psi^n_s$ is $(s,s)$-cyclic, and the lemma holds if $t = s$.}

Now suppose that $t > s$ and that the matroid $T^{\frac{(t-2)-s}{2}}(\Psi^n_s)$ is $(s, t-2)$-cyclic. Consider
$$T^{\frac{t-s}{2}}(\Psi^n_s)=T\left(T^{\frac{(t-2)-s}{2}}(\Psi^n_s)\right).$$
\blue{It follows from \Cref{truncation1} that each non-spanning circuit of $T^{\frac{(t-2)-s}{2}}(\Psi^n_s)$ is a circuit of $T^{\frac{t-s}{2}}(\Psi^n_s)$. Now, by \Cref{mat_rank},
\begin{align*}
r\left(T^{\frac{(t-2)-s}{2}}(\Psi^n_s)\right) & = {\textstyle \frac{n+s-(t-2)}{2}} \\
& \ge {\textstyle \frac{(s+t-2)+s-t+2}{2}} \\
& = s.
\end{align*}
Therefore, for all odd \bloo{$i\in[n]$}, we have that $\{e_i, e_{i+1}, \ldots, e_{i+s-1}\}$ is a non-spanning circuit of $T^{\frac{(t-2)-s}{2}}(\Psi^n_s)$, so is also an $s$-element circuit of $T^{\frac{t-s}{2}}(\Psi^n_s)$}. Furthermore, for all $j\in [n]$, if $\{e_j, e_{j+1}, \ldots, e_{j+t-3}\}$ and $\{e_{j+2}, e_{j+3}, \ldots, e_{j+t-1}\}$ are $(t-2)$-element cocircuits of $T^{\frac{(t-2)-s}{2}}(\Psi^n_s)$, then $\{e_j, e_{j+1}, \ldots, e_{j+t-1}\}$ is a $t$-element cocircuit of $T^{\frac{t-s}{2}}(\Psi^n_s)$. To see this, if $f$ is the element freely added to $T^{\frac{(t-2)-s}{2}}(\Psi^n_s)$, then it is easily checked that
$$\left(E\left(T^{\frac{(t-2)-s}{2}}(\Psi^n_s)\right)-\{e_j, e_{j+1}, \ldots, e_{j+t-1}\}\right)\cup \{f\}$$
is a hyperplane of the resulting matroid. Therefore
$$E\left(T^{\frac{t-s}{2}}(\Psi^n_s)\right)-\{e_j, e_{j+1}, \ldots, e_{j+t-1}\}$$
is a hyperplane of $T^{\frac{t-s}{2}}(\Psi^n_s)$, so $\{e_j, e_{j+1}, \ldots, e_{j+t-1}\}$ is a $t$-element cocircuit of $T^{\frac{t-s}{2}}(\Psi^n_s)$. Hence, by induction, $T^{\frac{t-s}{2}}(\Psi^n_s)$ is $(s, t)$-cyclic.
\end{proof}

\begin{proof}[Proof of Theorem~\ref{free}]
  Let $M$ be an $(s, t)$-cyclic matroid on $n$ elements, where \blue{$n \ge s+t-1$ and $t\ge s$.  Then, by \Cref{even_n},} $n$ is even, and $s\equiv t\bmod{2}$.
Let $\sigma=(e_1, e_2, \ldots, e_n)$ be an $(s, t)$-cyclic ordering of $E(M)$. Without loss of generality, we may assume that, for all odd \bloo{$i\in[n]$}, the set \bloo{$\sigma(i, i+s-1)$} is an $s$-element circuit of $M$. Now consider $\Psi^n_s$. To ease reading, we may assume that $E(M)=E(\Psi^n_s)$ and $\sigma=(e_1, e_2, \ldots, e_n)$ is an $(s, s)$-cyclic ordering of $\Psi^n_s$, where \bloo{$\sigma(i, i+s-1)$} is an $s$-element circuit of $\Psi^n_s$ for all odd \bloo{$i\in[n]$}. Note that the dual of $\Psi^n_s$ has no \bloo{loops}.

First suppose that $t=s$. \blue{By \Cref{truncation2}, both $M$ and $\Psi^n_s$ are $(s, s)$-cyclic matroids with $n$ elements. Therefore, by \Cref{set_rank}, for all $i \in [\tfrac{n}{2}]$ and $k$ such that $1 \leq k \leq m$, we have that
\begin{align*}
r_M\big(\sigma(x_i, y_i)\cup & \sigma(x_{i+1}, y_{i+1})\cup \cdots\cup \sigma(x_{i+k-1}, y_{i+k-1})\big) \\
& = r_{M^*[\mathscr I]}\big(\sigma(x_i, y_i)\cup \sigma(x_{i+1}, y_{i+1})\cup \cdots\cup \sigma(x_{i+k-1}, y_{i+k-1})\big),
\end{align*}
where $x_i=e_{2i-1}$ and $y_i=e_{2i+s-2}$ for all $i\in \{1, 2, \ldots, \frac{n}{2}\}$. Hence, by \Cref{weak_map}, under the identity map, $M$ is a weak-map image of $\Psi^n_s$.}

\blue{Now suppose $t > s$. By \Cref{truncation2}, the matroid $T^{\frac{t-s}{2}}(\Psi^n_s)$ is an $(s,t)$-cyclic matroid. It remains to} show that $M$ is a weak-map image of $T^{\frac{t-s}{2}}(\Psi^n_s)$. Let $I$ be an independent set in $M$. By \Cref{weak_map}, under the identity map, $M$ is a weak-map image of $\Psi^n_s$, and so $I$ is an independent set in $\Psi^n_s$. \bloo{From \Cref{mat_rank}, we have that}
\begin{align*}
r(M) & = {\textstyle \frac{n+s-t}{2}=\frac{n}{2}-\left(\frac{t-s}{2}\right)} = {\textstyle r(\Psi^n_s)-\left(\frac{t-s}{2}\right),}
\end{align*}
\bloo{Therefore,} $|I|\le r(\Psi^n_s)-\left(\frac{t-s}{2}\right)$. Therefore, as $T^{\frac{t-s}{2}}(\Psi^n_s)$ is the $\left(\frac{t-s}{2}\right)$-th truncation of $\Psi^n_s$, it follows by \Cref{truncation1} that $I$ is independent in $T^{\frac{t-s}{2}}(\Psi^n_s)$. In particular, under the identity map, $M$ is a weak-map image of $T^{\frac{t-s}{2}}(\Psi^n_s)$. This completes the proof of \Cref{free}.
\end{proof}

\section{Counterexample}
\label{counterexample}

In this section, we give a counterexample to a conjecture posed in \cite{bre19b}. Let $s$ be an integer exceeding two, and let $M$ be an $(s, s)$-cyclic matroid such that $|E(M)|\ge 2s+2$. A matroid $N$ is an {\em inflation} of $M$ if $N$ can be obtained from $M$ by first taking an elementary quotient in which none of the $s$-element cocircuits corresponding to consecutive elements in the cyclic ordering are preserved, \bloo{which produces an $(s,s+2)$-cyclic matroid}, and then taking an elementary lift in which none of the $s$-element circuits corresponding to consecutive elements in the cyclic ordering are preserved. The resulting matroid $N$ is $(s+2, s+2)$-cyclic. The conjecture in \cite[Conjecture~6.1]{bre19b} is the following:

\begin{conjecture}
Let $s$ be an integer exceeding two, and let $M$ be an $(s, s)$-cyclic matroid.
\begin{enumerate}[{\rm (i)}]
\item If $s$ is even, then $M$ can be obtained from a spike or a swirl by a sequence of inflations.

\item If $s$ is odd, then $M$ can be obtained from a wheel or a whirl by a sequence of inflations.
\end{enumerate}
\label{conjecture}
\end{conjecture}

\bloo{Now consider the matroid $\Psi^n_s$, where $s\ge 5$. If $\Psi^n_s$ can be obtained from a spike, swirl, wheel, or whirl by a sequence of inflations, then $\Psi^n_s$ is an elementary lift of some $(s-2,s)$-cyclic matroid, or, equivalently, using \Cref{selfdual}, $(\Psi^n_s)^* \cong \Psi^n_s$ is the elementary quotient of some $(s,s-2)$-cyclic matroid. We shall establish a counterexample to Conjecture~\ref{conjecture} by showing that no such $(s,s-2)$-cyclic matroid exists; in fact, we prove a more general result.}

\bloo{Let $M'$ be a rank-$(\frac{n}{2}+1)$ matroid in which there is a cyclic ordering $\left.\sigma=(e_1, e_2, \ldots, e_n)\right.$ of its ground set such that $\{e_i, e_{i+1}, \ldots, e_{i+s-1}\}$ is a circuit of $M$ for all odd $i\in [n]$. Further assume that $\sigma$ is also an $(s,s)$-cyclic ordering of $\Psi^n_s$ such that $\{e_i,e_{i+1},\ldots,e_{i+s-1}\}$ is a circuit of $\Psi^n_s$ for all odd $i \in [n]$. The following results show that $\Psi^n_s$ is not a quotient of $M'$. For the next lemma see, for example, \cite[Proposition 7.3.6]{ox11}.}

\begin{lemma}
Let $M_1$ and $M_2$ be matroids on the same ground set. Then $M_2$ is a quotient of $M_1$ if and only if every circuit of $M_1$ is a union of circuits of $M_2$.
\label{quotient}
\end{lemma}

Key to the counterexample shall be the following sets. Let $M$ be a matroid on $n$ elements and let $s$ be an integer exceeding three. Suppose that $\left.\sigma=(e_1, e_2, \ldots, e_n)\right.$ is a cyclic ordering of $E(M)$ such that, for all odd $i\in [n]$, the set $\{e_i, e_{i+1}, \ldots, e_{i+s-1}\}$ is an $s$-element circuit of $M$. For all odd $i \in [n]$, and for all integers $k$ and $\ell$ such that $2 \leq k,\ell \leq s-1$ and $s-1 \leq k+\ell \leq 2s-4$, define
\[C_{i, k, \ell} = \sigma(i, i+k-1) \cup \sigma(i+2k+\ell-s+2, i+2k+2\ell-s+1).\]
Informally, \blue{starting at $e_i$}, there are $k$ consecutive elements of $\sigma$ in $C_{i, k, \ell}$, followed by $k+\ell-(s-2)$ consecutive elements of $\sigma$ not in $C_{i, k, \ell}$, followed by $\ell$ consecutive elements of $\sigma$ in $C_{i, k, \ell}$.

The next lemma establishes that certain subsets of the ground set of $\Psi^n_s$ containing $C_{i, k, \ell}$ are circuits of $\Psi^n_s$. The subsequent lemma shows that these subsets are also circuits of $M'$. We will eventually combine these two lemmas to show that $\Psi^n_s$ is not a quotient of $M'$.

\begin{lemma}
Let $s$ be an integer exceeding three, and let $\sigma=(e_1, e_2, \ldots, e_n)$ be an $(s, s)$-cyclic ordering of $\Psi^n_s$ such that, for all odd $i\in [n]$, the set $\{e_i, e_{i+1}, \ldots, e_{i+s-1}\}$ is an $s$-element circuit of $\Psi^n_s$. Suppose that $n\ge 4s-8$. Then, for all odd $i\in [n]$, and for all $k$ and $\ell$ such that $2\le k, \ell\le s-1$ and $s-1\le k+\ell\le 2s-4$, the set $C_{i, k, \ell}\cup \{\bloo{x}\}$ is a circuit of $\Psi^n_s$, where $\bloo{x} \in \sigma(i+k,i+2k+\ell-s+1)$, and
$$\bloo{x}\neq
\begin{cases}
e_{i+k} & \mbox{if $k=s-1$}; \\
e_{i+2k+\ell-s+1} & \mbox{if $\ell=s-1$}.
\end{cases}$$
\label{psicircuits}
\end{lemma}

\begin{proof}
Recall the bipartite graph $G^n_s$ whose vertex parts are $\left.E=\{e_1, e_2, \ldots, e_n\}\right.$ and $\{1, 2, \ldots, \frac{n}{2}\}$ and, for all $i\in \{1, 2, \ldots, \frac{n}{2}\}$, the set of neighbours of $i$ is
$$\bloo{N(i)}=\{e_{2i-1}, e_{2i}, \ldots, e_{2i+s-2}\},$$
where subscripts are interpreted modulo $n$. \bloo{Let $i_0 = \tfrac{i+1}{2}$ and $\left.j_0 = \tfrac{i+2(k+\ell-s)+3}{2}\right.$.} Observe that
\begin{align*}
\bloo{N\left(i_0\right)}=\{e_i, e_{i+1}, \ldots, e_{i+k-1}, \ldots, e_{i+s-1}\}
\end{align*}
and
\begin{align*}
\bloo{N\left(j_0\right)}=\{e_{i+2(k+\ell-s)+2}, & e_{i+2(k+\ell-s)+3}, \ldots, \\
& e_{i+2k+\ell-s+2}, \ldots, e_{i+2(k+\ell)-s+1}\}.
\end{align*}
In particular, $\bloo{N\left(i_0\right)}\cup \bloo{N\left(j_0\right)}$ contains $C_{i, k, \ell}$. Also recall that $\Psi^n_s$ is the dual of the transversal matroid on $E$ in which
$$\bloo{(N(1), N(2), \ldots, N({\textstyle \frac{n}{2}}))}$$
is a presentation. 

We first show that $C_{i, k, \ell}\cup \{x\}$ is dependent in $\Psi^n_s$ by showing that $\left.E-(C_{i, k, \ell}\cup \{x\})\right.$ is not cospanning in $\Psi^n_s$. Consider $G^n_s$ and the subset \bloo{$\left[i_0,j_0\right]$} of $[\frac{n}{2}]$. \blue{Since $n \geq 4s-8$, we have that \bloo{$N\left(\left[i_0,j_0\right]\right) \neq E$}}, and so
$$\left|N\left(\bloo{\left[i_0,j_0\right]}\right)\right| = 2k+2\ell-s+2.$$
Therefore, as $C_{i, k, \ell}\cup \{e_j\}\subseteq N\left(\left[\bloo{i_0,j_0}\right]\right)$ and $|C_{i, k, \ell}\cup \{x\}|=k+\ell+1$, it follows that
\begin{align*}
\Big|N\left(\left[{\textstyle i_0, j_0}\right]\right) \left. - \right. & (C_{i, k, \ell}\cup \{x\})\Big| \\
& = \left|N\left(\left[{\textstyle i_0,j_0}\right]\right)\right|-\left|C_{i, k, \ell}\cup \{x\}\right| \\
& = (2k+2\ell-s+2)-(k+\ell+1) \\
& = k+\ell-s+1 \\
& < k+\ell-s+2 \\
& = \left|\left[{\textstyle i_0,j_0}\right]\right|.
\end{align*}
Hence, by Hall's Theorem~\cite{hal35}, $E-(C_{i, k, \ell}\cup \{x\})$ is not cospanning in $\Psi^n_s$. Thus $C_{i, k, \ell}\cup \{x\}$ is dependent in $\Psi^n_s$.

Since $C_{i, k, \ell}\cup \{x\}$ is dependent, $C_{i, k, \ell}\cup \{x\}$ contains a circuit $C$ of $\Psi^n_s$. If $|C|=|E|-\frac{n}{2}+1=\frac{n}{2}+1$, then, as $n \geq 4s-8$, it follows that $|C| \geq 2s-3$. Furthermore, $|C| \leq |C_{i, k, \ell}\cup \{x\}| = k+\ell+1 \leq 2s-3$. Thus $C=C_{i, k, \ell}\cup \{x\}$, and so $C_{i, k, \ell}\cup \{x\}$ is a circuit of $\Psi^n_s$. Therefore, by Lemma~\ref{circuits}, we may assume that there are $i_1, j_1\in [\frac{n}{2}]$ satisfying (i)--(v) of that lemma. If $i_1=j_1$, then, by Lemma~\ref{circuits}(ii), \bloo{$C=N(i_1)$} for some $i_1\in [\frac{n}{2}]$. Now, $C$, and thus $C_{i,k,\ell} \cup \{x\}$, contains $s$ consecutive elements of $\sigma$. But if $C_{i,k,\ell} \cup \{x\}$ contains $s$ consecutive elements, then $k+\ell = s-1$, in which case $C_{i,k,\ell} \cup \{x\}$ is a circuit, and we are done. Therefore $i_1\neq j_1$, and so, by Lemma~\ref{circuits}\bloo{(iii) and~(iv)},
\begin{align}\label{subset1}
N([i_1, j_1])-N([i_1+1, j_1])=\bloo{\{e_{2i_1-1},e_{2i_1}\}}\subseteq C
\end{align}
and
\begin{align}\label{subset2}
N([i_1, j_1])-N([i_1, j_1-1])\bloo{=\{e_{2j_1+s-3},e_{2j_1+s-2}\}}\subseteq C.
\end{align}

\bloo{Since $\ell \leq s-1$ and $|N([i_1,j_1])|>s$, we have that $e_{2i_1-1}$ is not contained in the set of $\ell+1$ consecutive elements $\sigma(i+2k+\ell-s+1,i+2k+2\ell-s+1)$. Also, it follows from (\ref{subset1}) that $e_{2i_1-1} \neq x$. Therefore,
\[e_{2i_1-1} \in \sigma(i,i+k-1)\]
and so $i_1 = i_0 + i'$ for some $0 \leq i' \leq \lfloor{\frac{k}{2}}\rfloor$. Symmetrically,
\[e_{2j_1+s-2} \in \sigma(i+2k+\ell-s+2,i+2k+2\ell-s+1)\]
and so $j_1 = j_0 - j'$ for some $0 \leq j' \leq \lceil{\frac{\ell}{2}}\rceil$.}

\bloo{By \Cref{circuits}(ii), $C$ is a subset of $N([i_1,j_1])$ containing $\left.|N([i_1,j_1])| - |[i_1,j_1]| + 1\right.$ elements. Now,
\[|N([i_1,j_1])| = |N([i_0,j_0])| - 2(i'+j') = 2k+2\ell-s+2 - 2(i'+j'),\]
and
\[|[i_1,j_1]| = |[i_0,j_0]| - (i'+j')=k+\ell-s+2 - (i'+j')\]
so
\begin{align} \label{circuitsize1}
|C| = k+\ell+1-(i'+j').
\end{align}
On the other hand,
\begin{align} \label{circuitsize2}
|C| \leq |(C_{i,k,\ell}\cup\{x\}) \cap N([i_1,j_1])| = k+\ell+1 - 2(i'+j').
\end{align}}

\bloo{Therefore, since both (\ref{circuitsize1}) and (\ref{circuitsize2}) hold, we have that $i'=j'=0$, that is $i_0=i_1$ and $j_0=j_1$, and that $|C| = |C_{i,k,\ell}\cup\{x\}|$. Hence, $C=C_{i,k,\ell}\cup\{x\}$, completing the proof of the lemma.}
\end{proof}


\begin{lemma}
\label{quotient_circuits}
Let $n \geq 4s-8$, and suppose that $\Psi^n_s$ is a quotient of $M'$. Then, for all odd $i \in [n]$, and for all $k$ and $\ell$ such that $2 \leq k,\ell \leq s-1$ and $s-1 \leq k+\ell \leq 2s-4$, the set $C_{i, k, \ell} \cup \{\bloo{x}\}$ is a circuit of $M'$, where \bloo{$x \in \sigma(i+k,i+2k+\ell-s+1)$}, and
\bloo{$$x \neq
\begin{cases}
e_{i+k} & \mbox{if $k=s-1$}; \\
e_{i+2k+\ell-s+1} & \mbox{if $\ell=s-1$}.
\end{cases}$$}
\end{lemma}

\begin{proof}
Since $\Psi^n_s$ is a quotient of $M'$, it follows by Lemma~\ref{quotient} that every circuit of $M'$ is a union of circuits of $\Psi^n_s$. Now, by Lemma~\ref{psicircuits}, $C_{i, k, \ell}\cup \{x\}$ is a circuit of $\Psi^n_s$. Therefore, to prove the lemma, it suffices to show that $M'$ has a circuit contained in $C_{i, k, \ell} \cup \{x\}$. The proof is by induction on $k + \ell$.

If $k+\ell = s-1$, then
$$C_{i, k, \ell} = \sigma(i, i+k-1)\cup \sigma(i+k+1, i+s-1).$$
Therefore, $x = e_{i+k}$, and $C_{i, k, \ell}\cup \{x\} = \sigma(i, i+s-1)$ which is a circuit of $M'$. Furthermore, if $k+\ell = s$, then
$$C_{i, k, \ell} = \sigma(i, i+k-1)\cup \sigma(i+k+2, i+s+1),$$
so $C_{i, k, \ell}\cup \{x\} = \sigma(i, i+s+1) - \{y\}$, where $y$ is the element of $\{e_{i+k}, e_{i+k+1}\}$ which is not equal to $x$. Since $y\in \sigma(i, i+s-1)\cap \sigma(i+2, i+s+1)$, it follows by circuit elimination that $M'$ has a circuit contained in $C_{i, k,\ell} \cup \{x\}$, as desired.

Now suppose that the lemma holds for all $2\le k', \ell'\le s-1$ and $\left.s-1\le k'+\ell'\le 2s-4\right.$ such that $k'+\ell' = k+\ell-1$. First assume that either $k$ or $\ell$ is equal to $s-1$. If $k=s-1$, then $x\neq e_{i+s-1}$. Therefore, by the induction assumption,
$$C_{i+2, k-1, \ell}\cup \{x\}=\sigma(i+2, i+s-1)\cup \{x\}\cup \sigma(i+\ell+s, i+2\ell+s-1)$$
is a circuit of $M'$. Thus, by circuit elimination between $C_{i+2, k-1, \ell}\cup \{x\}$ and $\sigma(i, i+s-1)$ on $e_{i+s-1}$, the matroid $M'$ has a circuit contained in
\begin{align*}
\sigma(i, i+s-2)\cup \{x\}\cup \sigma(i+\ell+s, i+2\ell+s-1) & = C_{i, s-1, \ell}\cup \{x\} \\
& = C_{i, k, \ell}\cup \{x\}
\end{align*}
as desired. A similar argument shows the lemma holds if $\ell=s-1$.

We may now assume that neither $k$ nor $\ell$ is equal to $s-1$. Furthermore, since $k+\ell \geq s+1$, we have that $k \neq 2$ and $\ell \neq 2$. Assume $k = \ell = 3$. This implies that $s = 5$, so
$$C_{i, k, \ell} = C_{i, 3, 3} = \{e_i, e_{i+1}, e_{i+2}, e_{i+6}, e_{i+7}, e_{i+8}\}.$$
By the induction assumption, if $x \in \{e_{i+4}, e_{i+5}\}$, then the desired result follows from circuit elimination between 
\[C_{i, 3, 2} \cup \{e_{i+4}\} = \{e_i, e_{i+1}, e_{i+2}, e_{i+4}, e_{i+5}, e_{i+6}\}\]
and $\{e_{i+4}, e_{i+5}, e_{i+6}, e_{i+7}, e_{i+8}\}$. If $j = i+3$, then the result follows from circuit elimination between 
\[C_{i+2, 2, 3} \cup \{e_{i+4}\} = \{e_{i+2}, e_{i+3}, e_{i+4}, e_{i+6}, e_{i+7}, e_{i+8}\}\]
and $\{e_i, e_{i+1}, e_{i+2}, e_{i+3}, e_{i+4}\}$.

Lastly, assume that either $k\ge 4$ or $\ell\ge 4$, which implies $s\geq 6$. We establish that the lemma holds when $k\ge 4$. The proof of the lemma when $\ell\ge 4$ is similar and omitted. Assume $k\geq 4$. \bloo{Suppose} $x\neq e_{i+2k+\ell-s+1}$, that is $x\in \sigma(i+k, i+2k+\ell-s)$. Then, by the induction assumption, the set 
\[C_{i, k, \ell-1} \cup \{x\} = \sigma(i, i+k-1) \cup \{x\} \cup \sigma(i+2k+\ell-s+1, i+2k+2\ell-s-1)\] 
is a circuit. \bloo{If $\ell=s-2$ and $x = e_{i+2k+\ell-s}$, then the set
\[\sigma(i+2k+\ell-s,i+2k+2\ell-s+1)=\sigma(i+2k-2,i+2k+s-3)\]
is an $s$-element circuit of $M'$. Hence, circuit elimination between this circuit and $C_{i,k,\ell-1}\cup\{x\}$ on the element $e_{i+2k+\ell-s+1}$ gives a ciruit of $M'$ contained in
\[\sigma(i, i+k-1) \cup \{e_{i+2k+\ell-s}\} \cup \sigma(i+2k+\ell-s+2, i+2k+2\ell-s+1) = C_{i, k, \ell} \cup \{x\}\]
as desired. Otherwise, since $k \geq 4$, the set}
\[C_{i+2, k-2,\ell+1} \cup \{x\} = \sigma(i+2, i+k-1) \cup \{x\} \cup \sigma(i+2k+\ell-s+1, i+2k+2\ell-s+1)\]
is a circuit. \bloo{Again, circuit elimination between this circuit and $C_{i,k,\ell-1}\cup\{x\}$ on the element $e_{i+2k+\ell-s+1}$} implies that $M'$ has a circuit contained in
\[\sigma(i, i+k-1) \cup \{x\} \cup \sigma(i+2k+\ell-s+2, i+2k+2\ell-s+1) = C_{i, k, \ell} \cup \{x\}\]
as desired. 

\bloo{The final case to consider is when $x = e_{i+2k+\ell-s+1}$. By the induction assumption, and since $k \neq s-1$, the set}
\[
C_{i,k,\ell-1} \cup \{e_{i+k}\} = \sigma(i,i+k-1) \cup \{e_{i+k}\} \cup \sigma(i+2k+\ell-s+1,i+2k+2\ell-s-1)
\]
is a circuit of $M'$. Additionally, since $k \geq 4$, the set
\[
C_{i+2,k-2,\ell+1} \cup \{e_{i+k}\} = \sigma(i+2,i+k-1) \cup \{e_{i+k}\} \cup \sigma(i+2k+\ell-s+1,i+2k+2\ell-s+1)
\] is a circuit of $M'$. Circuit elimination between these circuits on the element $e_{i+k}$ implies that $M'$ has a circuit contained in
\[
\sigma(i,i+k-1) \cup \sigma(i+2k+\ell-s+1,i+2k+2\ell-s+1) = C_{i,k,\ell} \cup \{e_{i+2k+\ell-s+1}\}.
\]
This completes the proof of the case when $k\ge 4$, and thus completes the proof of the lemma.
\end{proof}

\begin{proposition}
Let $n \geq 4s-8$, where $s$ is an integer exceeding three. Then $\Psi^n_s$ is not a quotient of $M'$.
\end{proposition}

\begin{proof}
\bloo{Suppose $\Psi^n_s$ is a quotient of $M'$. We establish a contradiction by showing that $r(M') \leq \frac{n}{2}$}. By definition of $M'$, $\{e_1, e_2, \ldots, e_s\}$ is a circuit with rank $s-1$. The element $e_{s+1}$ may or may not be in the closure of $\{e_1, e_2, \ldots, e_s\}$, so $r(\{e_1, e_2, \ldots, e_{s+1}\}) \leq s$. Since $\{e_3, e_4, \ldots, e_{s+2}\}$ is a circuit, $e_{s+2} \in \cl(\{e_1, e_2, \ldots, e_{s+1}\})$, that is $r(\{e_1, e_2, \ldots, e_{s+2}\}) \leq s$. Repeating this process, we see that $r(\{e_1, e_2, \ldots, e_{s+2u}\}) \leq s-1+u$ for all $u \le \frac{n-s}{2}$. In particular, when $u = \frac{n}{2} - s + 1$, we have that $r(\{e_1, e_2, \ldots, e_{n-s+2}\}) \leq \frac{n}{2}$. However, by \Cref{quotient_circuits} with $i = n-2s+5$ and $k=\ell=s-2$, the set
$$\{e_{n-2s+5}, e_{n-2s+6}, \ldots, e_{n-s+2}\} \cup \{x\} \cup \{e_1, e_2, e_3, \ldots, e_{s-2}\}$$
is a circuit for all $x \in \{e_{n-s+3}, e_{n-s+4}, \ldots, e_{n-1}, e_n\}$, and so $\{e_1,e_2,\ldots,e_{n-s+2}\}$ is spanning. This implies $r(M') \leq \frac{n}{2}$, a contradiction.
\end{proof}

\section*{Acknowledgements}
We thank the referees for their careful reading of the paper, and their constructive and helpful comments.

\end{document}